\documentclass[10pt]{amsart}

\usepackage{geometry,graphicx,amssymb,amsmath,amsbsy,eucal,amsfonts,mathrsfs,amscd,bm}
\usepackage[all]{xy}
\usepackage{tikz-cd}
\usetikzlibrary{arrows,shapes,calc,snakes}
\usetikzlibrary{shapes,snakes}
\usepackage{tikz-3dplot}
\usetikzlibrary{intersections}
\usepackage{tkz-euclide}

\numberwithin{equation}{section}

\allowdisplaybreaks[3]

\newtheorem{theorem}{Theorem}[section]
\newtheorem{lemma}[theorem]{Lemma}
\newtheorem{assumption}[theorem]{Assumption}

\newtheorem{proposition}[theorem]{Proposition}
\theoremstyle{definition}
\newtheorem{definition}[theorem]{Definition}
\theoremstyle{remark}

\newcommand{\p}{{\partial}}

\newcommand{\nab}{\nabla}
\newcommand{\nonum}{\nonumber}

\newcommand{\mct}{\mathcal{T}_h}

\newcommand{\jump}[1]{\left[\hspace{-0.025in}\left[#1\right]\hspace{-0.025in}\right]}

\newcommand{\curl}{{\ensuremath\mathop{\mathrm{curl}\,}}}
\newcommand{\dive}{{\ensuremath\mathop{\mathrm{div}\,}}}
\newcommand{\Div}{{\rm div}\,}
\newcommand{\bcurl}{{\bld{\curl}}}

\newcommand{\pol}{\EuScript{P}}
\newcommand{\bpol}{\boldsymbol{\pol}}
\newcommand{\bld}[1]{\boldsymbol{#1}}
\newcommand{\bt}{\bld{t}}

\newcommand{\bI}{\bld{I}}
\newcommand{\bA}{\bld{A}}

\newcommand{\bv}{\bld{v}}
\newcommand{\bw}{\bld{w}}

\newcommand{\bp}{\bld{p}}
\newcommand{\bn}{\bld{n}}
\newcommand{\bu}{\bld{u}}
\newcommand{\bJ}{\bld{J}}

\newcommand{\bV}{\bld{V}}

\newcommand{\bPi}{\bld{\Pi}}

\newcommand{\bsigma}{\bld{\sigma}}
\newcommand{\bH}{\bld{H}}

\newcommand{\bL}{\bld{L}}

\newcommand{\btau}{{\bm \tau}}

\newcommand{\bbeta}{{\bm \beta}}

\newcommand{\bpsi}{\bm \psi}

\newcommand{\bbR}{\mathbb{R}}

\newcommand{\calV}{\mathcal{V}}

\newcommand{\calE}{\mathcal{E}}

\newcommand{\PST}{\mathcal{T}_h^{\rm ps}}
\newcommand{\CTT}{\mathcal{T}_h^{\rm ct}}
\newcommand{\rot}{{\rm rot}\,}
\newcommand{\Sh}{\mathcal{S}_h}

\newcommand{\calS}{\mathcal{S}}
\newcommand{\mbV}{\mathring{\bV}}
\newcommand{\hQ}{\mathring{Q}_h}
\newcommand{\hbV}{\mbV(Q_h)}
\newcommand{\calQ}{\mathcal{Q}}
\newcommand{\Ned}{\rev{N\'{e}d\'{e}lec} }

\newcommand{\MJN}[1]{{\color{black}#1}}

\newcommand{\rev}[1]{{\color{black}#1}}
\newcommand{\revj}[1]{{\color{black}#1}}

\newcommand\DB{\begin{color}{black}}
\newcommand\BD{\end{color}}

\title[FEM for Eigenvalues]{Convergence of Lagrange finite elements for the Maxwell Eigenvalue Problem in 2D}

\begin{document}

\author[]{Daniele Boffi}
\address{King Abdullah University of Science and Technology, Saudi Arabia and
University of Pavia, Italy}
\email{daniele.boffi@kaust.edu.sa}
\thanks{The first author is a member of the INdAM Research group GNCS and his
research is partially supported by IMATI/CNR and by PRIN/MIUR.
The second author is partially supported by 
the NSF grant DMS-1620100.  
The third author is partially supported by 
the NSF grants DMS-1719829 and DMS-2011733.}

\author[]{Johnny Guzm\'an}
\address{Division of Applied Mathematics,
Brown University
}
\email{johnny\_guzman@brown.edu}

\author[]{Michael Neilan}
\address{Department of Mathematics, University of Pittsburgh}
\email{neilan@pitt.edu}

\maketitle
\begin{abstract}
We consider finite element approximations of the Maxwell eigenvalue problem in two dimensions. 
We prove, in certain settings, convergence 
of the discrete eigenvalues using Lagrange finite elements.
In particular, we prove convergence in three scenarios:  piecewise linear elements
on Powell--Sabin triangulations,
piecewise quadratic elements on Clough--Tocher triangulations,
and piecewise quartics (and higher) elements on general shape-regular triangulations.
We provide numerical experiments that support the theoretical
results.  The computations also show that, on general triangulations,
the eigenvalue approximations are very sensitive to nearly singular vertices, i.e.,
vertices that fall on exactly two ``almost''  straight lines.
\end{abstract}

\thispagestyle{empty}

\section{Introduction}

Let $\Omega \subset \bbR^2$ be a contractible polygonal domain and consider the eigenvalue problem:
\rev{Find $\bu\in \bH_0({\rm rot},\Omega)$ such that}
\begin{equation}\label{weak}
(\rot \bu, \rot \bv)  = \eta^2 (\bu, \bv) \quad \forall \bv \in \bH_0({\rm rot}, \Omega),
\end{equation}
where  $\bH({\rm rot}, \Omega):=\{ \bv \in \bL^2(\Omega): \rot \bv \in L^2(\Omega)\}$,
 $\bH_0({\rm rot}, \Omega):=\{ \bv \in \bH({\rm rot}, \Omega): \bv \cdot \bt=0 \text{ on } \partial \Omega\}$,
 and $(\cdot,\cdot)$ denotes the $L^2$ inner product over $\Omega$.
Given a finite element space $\mbV_h \subset \bH_0({\rm rot}, \Omega)$,
a finite element method seeks $\bu_h\in \mbV_h\backslash \{0\}$
and $\eta_h\in \bbR$ satisfying
\begin{equation}\label{fem}
(\rot \bu_h, \rot \bv_h)  = \eta_h^2 (\bu_h, \bv_h) \quad \forall \bv_h \in \mbV_h.
\end{equation}
For example, one can take $\mbV_h$ to be the $\bH_0({\rm rot};\Omega)$-conforming
\Ned finite elements (i.e., the rotated Raviart-Thomas finite elements)
as the finite element space.  It is well-known  this choice leads
to a convergent approximation of the eigenvalue problem.  On the other hand,
taking $\mbV_h$ as a space continuous piecewise polynomials (i.e., a $\bH^1(\Omega)$-conforming
Lagrange finite element) may lead to spurious eigenvalues for any mesh
parameter.

There is a vast literature on this subject. The interested reader is referred
to~\cite[Section~20]{BoffiActa} for an extensive survey including a
comprehensive list of references about \Ned finite elements and
to~\cite{bfgp,BoffiEtal00} for a discussion about the use of standard Lagrange
finite elements (see also~\cite{ArnoldEtal10} for a discussion of these
phenomena in the context of the finite element exterior calculus).  
\revj{Many formulations have been developed based on penalization and/or regularization 
(e.g., \cite{CoDa02,BuffaEtal2009,BonitoGuermond11,BadiaCordina12,Duan, Duan19B,Duan20}), showing  Lagrange elements can lead to consistent approximations to \eqref{weak}. However, we are not aware of a previous analysis of Lagrange elements 
on macro elements using the standard formulation \eqref{fem}, and this is the main objective of this work. }

To better appreciate the problem and its discretization, we consider the equivalent formulation
introduced in~\cite{bfgp}
for $\eta \neq 0$: \rev{$(\bsigma,p)\in \bH_0({\rm rot},\Omega)\times L^2_0(\Omega)$ such that}
\begin{subequations}
\label{mixedweak}
\begin{alignat}{2}
(\bsigma, \btau)+(p, \rot \btau)= &0 \quad && \forall \btau \in \bH_0({\rm rot}, \Omega), \\
(\rot \bsigma, q)=&-\lambda (p,q)  \quad && \forall q \in L_0^2(\Omega).
\end{alignat}
\end{subequations}
Taking $q = \rot \bv$ with $\bv\in \bH_0({\rm rot},\Omega)$ shows
the equivalence of \eqref{mixedweak} and \eqref{weak} with
 $\bsigma=\bu,  \lambda=\eta^2$, and  $p= -\frac{1}{\lambda} \rot \bu$.

The corresponding finite element method for the mixed formulation \eqref{mixedweak}
seeks  $\bsigma_h \in \mbV_h\backslash \{0\}$, $ p_h \in \hQ$, and $ \lambda_h \in \mathbb{R}$ such that 
\begin{subequations}
\label{mixedfem}
\begin{alignat}{2}
(\bsigma_h, \btau_h)+(p_h, \rot \btau_h)= &0, \quad && \forall \rev{\btau_h} \in  \mbV_h,  \\
(\rot \bsigma_h, q_h)=&-\lambda_h (p_h,q_h) \quad &&  \forall q_h \in \hQ
\end{alignat}
\end{subequations}
with $\hQ \subset L_0^2(\Omega)$.  Similar to the continuous problem, 
if the finite element spaces
satisfy $\rot \mbV_h \subset \hQ$, then the mixed finite element formulation \eqref{mixedfem}
is equivalent to the primal one \eqref{fem} with  $\bsigma_h=\bu_h,  \lambda_h=\eta^2_h$ and $p_h= -\frac{1}{\rev{\lambda_h}} \rot \bu_h$.

If $\mbV_h$ is the \Ned space of index $k$, then we may take $\hQ$ to be the space
of piecewise polynomials of degree $k-1$. 
In this case, $(\mbV_h, \hQ)$ forms an inf-sup stable pair of spaces, in particular, 
there exists a Fortin projection 
\MJN{
\[
\bPi_V: \mbV\to \mbV_h
\]
satisfying
}
\begin{subequations}\label{fortin0}
\begin{alignat}{2}
\rot \bPi_V \btau= & \Pi_Q \rot \btau \quad && \forall  \btau \in \mbV,\\
\|\bPi_V \btau-\btau\|_{L^2(\Omega)} \le & C h^{\frac{1}{2}+\delta} (\|\btau\|_{H^{\frac{1}{2}+\delta}(\Omega)}+ \|\rot \btau\|_{L^2(\Omega)})  \quad &&  \forall \btau \in \mbV.
\end{alignat}
\end{subequations}
Here $\mbV:=\bH_0(\rot, \Omega) \cap \bH({\rm div},\Omega)$. \MJN{Moreover, $\delta\in (0,\frac{1}{2}]$ is a parameter such that $\mbV\hookrightarrow \bH^{\frac{1}{2}+\delta}(\Omega)$ \cite{VectorP}, and }
$\Pi_Q: L_0^2(\Omega) \rightarrow \hQ$ is the $L^2$ orthogonal projection onto $\hQ$. 
Using this projection one can prove that the corresponding source problems
\DB
converges
\BD
uniformly, and this is sufficient to prove convergence of the eigenvalue problem \eqref{fem}
(see \cite{Kato95,BoffiActa} and Proposition \ref{uniform}).

On the other hand, if $\mbV_h$ is taken to be the Lagrange finite element space
of degree $k$, then a natural choice of $\hQ$ is the space of (discontinuous) piecewise polynomials of degree $k-1$.
However, $(\mbV_h, \hQ)$ is {\em not} inf-sup stable on generic triangulations, at least when $k=1$ \cite{QinThesis,BoffiEtal06},
 and therefore there does not exist a Fortin projection satisfying \eqref{fortin0}. 
On the other hand, the pair $(\mbV_h, \hQ)$ is known to be stable on special
triangulations, even if the inf-sup condition might not be sufficient to
guarantee the existence of a Fortin projector satisfying~\eqref{fortin0}
(see~\cite{BoffiEtal00}).

Wong and Cendes \cite{WongCendes88} showed numerically that,
on very special triangulations, solutions to \eqref{fem} do converge
 to the correct eigenvalues using piecewise linear Lagrange elements (i.e., $k=1$). 
 In fact, they used precisely the Powell--Sabin triangulations (see Figure \ref{fig:PS}).  \revj{A rigorous proof of this result 
 has remained unsettled until now; see the review paper \cite{BoffiActa} for a discussion.}
\revj{Specifically},  we prove that using Lagrange elements in conjunction with Powell--Sabin triangulation leads to a convergent method. We do this by proving that there is a Fortin projection of sorts. We show that there exists an operator $\bPi_V: \hbV \rightarrow \mbV_h$  satisfying 
\begin{subequations}\label{fortin}
\begin{alignat}{2}
\rot \bPi_V \btau= & \rot \btau \quad && \forall  \btau \in \hbV, \\
\|\bPi_V \btau-\btau\|_{L^2(\Omega)} \le & C h^{\frac{1}{2}+\delta} (\|\btau\|_{H^{\frac{1}{2}+\delta}(\Omega)}+ \|\rot \btau\|_{L^2(\Omega)}), \quad && \forall  \btau \in \hbV,
\end{alignat}
\end{subequations}
where $\hbV=\{ \bv \in \mbV  : \rot  \bv \in  \hQ\}$. Note that \eqref{fortin0} implies \eqref{fortin}, and 
we prove convergence of the eigenvalue problem whenever there is a projection $\bPi_V$ satisfying \eqref{fortin}. In addition to 
linear Lagrange elements on Powell--Sabin triangulations, we 
prove the existence of such a projection on Clough--Tocher splits using quadratic Lagrange elements,  and on general triangulations using 
$k$th degree Lagrange elements with $k \ge 4$ (i.e., the Scott--Vogelius finite elements). For the Scott--Vogelius finite elements, we 
find the approximate eigenvalues are extremely sensitive if the mesh has nearly singular vertices, i.e.,
vertices that fall on exactly two ``almost'' straight lines (cf.~Section \ref{sec-SV}).  We give numerical examples that illustrate this behavior.

\revj{The analysis of composite triangulations (e.g. Clough-Tocher and Powell-Sabin) on the problem \eqref{weak} goes back at least to the work of Costabel and Dauge \cite{CoDa02}. Recently,  Duan et al. \cite{Duan,Duan19B,Duan20} considered Lagrange finite elements for Maxwell's eigenvalue problem in two and three dimensions using composite triangulations.  However, as noted earlier, they use a different formulation than the standard one \eqref{fem}.  In particular, in \cite{Duan20}  they add a Lagrange multiplier and an equation of the form appears $(\dive \bu_h, q_h)=0$ which can be thought of as a Kikuchi-type formulation \cite{Kikuchi}, where one transfers the derivatives to $\bu_h$. In \cite{Duan19B} a similar formulation is used with a regularization term. 

As mentioned above, the main idea 
to show convergence of Lagrange elements using the standard formulation \eqref{weak} on certain triangulations is 
the construction of a Fortin-type operation with certain approximation properties.
On certain composite triangulations (e.g. Powell-Sabin, Clough-Tocher, Alfeld, Worsey-Farin) exact sequences and/or Fortin projections have been developed; see for example \cite{GLN19, GLN20, FGN20, Zhang3d, QinThesis, ChristiansenHu}. These results have led to stable finite element for fluid flow problems; see for example \cite{Neilan}. In this paper, for the Powell-Sabin and Clough-Tocher triangulations, we cannot directly use  the Fortin projections defined in \cite{GLN19, FGN20} since they require too much smoothness. Instead, we pre-process with a Scott-Zhang type interpolant that preserves the vanishing tangential components, and then use the degrees of freedom  in \cite{GLN19, FGN20}. These  projections are sufficient for our purposes, however, it would be very interesting to see if one can construct $L^2$ bounded commuting projection for these sequences as is done in the finite element exterior calculus (FEEC) \cite{Christiansen}. 
If bounded $L^2$ bounded commuting projections exist, then
the convergence of eigenvalue problems  follows from the general theory in FEEC \cite{AFW2006, ArnoldEtal10, BoffiActa}. 
}


The paper is organized as follows: In the next section we give  a convergence
proof for finite elements spaces with stable projections. In \DB
Section~\ref{sec-Examples},\BD{}
we provide three examples of Lagrange finite element spaces with stable projections:  the piecewise linear Lagrange space
on Powell--Sabin splits, the piecewise quadratic Lagrange space on Clough--Tocher splits,
and the piecewise $k$th degree Lagrange space on generic triangulations.
\rev{In Section \ref{sec-Numerics} we provide numerical experiments
and make some concluding remarks in Section \ref{sec-conclude}.}

\section{Convergence Framework}
Define the two-dimensional $\bcurl$, ${\rm rot}$, and divergence operators as
\[
\bcurl u = \big(\frac{\p u}{\p x_2}, -\frac{\p u}{\rev{\p x_1}}\big)^\intercal,\qquad
\rot \bv  = \frac{\p v_2}{\p x_1} - \frac{\p v_1}{\p x_2},\qquad \MJN{\Div \bv = \frac{\p v_1}{\p x_1} + \frac{\p v_2}{\p x_2}},
\]
\MJN{and define the Hilbert spaces
\begin{align*}
\bH_0({\rm rot},\Omega) &= \{\bv\in \bL^2(\Omega):\ \rot \bv\in L^2(\Omega),\ \bv\cdot \bt|_{\p\Omega} = 0\},\\
\bH({\rm div},\Omega) & = \{\bv\in \bL^2(\Omega):\ \Div \bv\in L^2(\Omega)\},
\end{align*}
where $\bt$ is a unit tangent vector of $\p \Omega$.
Recall that $\mbV = \bH_0({\rm rot},\Omega)\cap \bH({\rm div},\Omega)$.
}

Let $\mbV_h\subset \bH_0({\rm rot},\Omega)$ and $\hQ \subset L_0^2(\Omega)$ be finite element spaces such that $ \rot \mbV_h \subset \hQ$.

\subsection{Source problems}
We will require the corresponding source problems for the analysis. 
\MJN{To this end, we define the solution operators $\bA:L^2(\Omega)\to \bH_0({\rm rot},\Omega)$
and $T:L^2(\Omega)\to L^2_0(\Omega)$ such that for given $f\in L^2(\Omega)$, there holds}
\begin{subequations}\label{source}
\begin{alignat}{2}
(\bA f, \btau)+(T f, \rot \btau)&= 0, \quad && \forall \btau \in \bH_0({\rm rot}, \Omega), \\
(\rot \bA f, q)&=(f,q)  \quad && \forall q \in L_0^2(\Omega).
\end{alignat}
\end{subequations}

Likewise, the discrete source problem is given by:
\begin{subequations}\label{femsource}
Find $\bA_h f \in \mbV_h$ and $T_h f \in \hQ$ such that
\begin{alignat}{2}
(\bA_h f, \btau_h)+(T_h f, \rot \btau_h)&= 0 \quad && \forall \btau \in  \mbV_h,  \\
(\rot \bA_h f, q_h)&=(f,q_h) \quad &&  \forall q_h \in \hQ.
\end{alignat}
\end{subequations}
Note that $\bA f=\bcurl Tf$, and so $\dive \bA f=0$. Moreover, using that $\rot \bA f=f$ we have that $\rev{\bA}f \in \mbV$. 

We define the operator norm:
\begin{equation}
\|T-T_h\|:=\sup_{f \in L^2(\Omega)\backslash \{0\}} \frac{\|(T-T_h) f\|_{L^2(\Omega)}}{\|f\|_{L^2(\Omega)}}.
\end{equation}
We will use the next standard result that states that the uniform convergence of the
discrete source problem implies convergence of the discrete eigenvalues.
\DB
This result is a consequence of the classical discussion
in~\cite[Section~8]{BaOs} (see also~\cite[Section~9]{BoffiActa}
and~\cite[Theorem~4.4]{bfgp}).
\BD
\begin{proposition}\label{uniform}
Let $T$ and $T_h$ be defined from \eqref{source} and \eqref{femsource}, respectively, and
  suppose that $\|T-T_h\| \rightarrow 0$ as $h \rightarrow 0$. Consider the problem \eqref{mixedweak}  
  and consider the nonzero eigenvalues $0 < \lambda^{(1)} \le \lambda^{(2)} \le \cdots$. Consider also  \eqref{mixedfem} and its non-zero eigenvalues  $0< \lambda_h^{(1)} \le \lambda_h^{(2)} \le \cdots$. Then, for any fixed $i$, $\lim_{h \rightarrow 0} \lambda_h^{(i)} =\lambda^{(i)}$.
\end{proposition}

Therefore, to prove convergence of eigenvalues it suffices to show uniform convergence of the 
discrete source problem. To prove this, we will exploit the 
embedding $\mbV \hookrightarrow \bH^{\frac{1}{2}+\delta}(\Omega)$
along with an assumption on the finite element spaces. The embedding result is proved in three dimensions in \cite{VectorP}, and we state the two dimensional version here.

\begin{proposition}\label{embedding}
Let $\Omega$ be a contractible polygonal domain. Then 
there exists constants $\delta\in (0,\frac{1}{2}]$ and $C>0$ such that 
\begin{equation*}
\|\bv\|_{H^{\frac{1}{2}+\delta}(\Omega)} \le  C  (\|\dive \bv\|_{L^2(\Omega)}+ \|\rot \bv\|_{L^2(\Omega)}) \quad \forall \bv \in \mbV.
\end{equation*} 
\end{proposition}
From now on $\delta$ will refer to the delta of the above proposition. We will use the following  space
\begin{equation}\label{eqn:FunkySpace}
{\hbV=  \{\btau\in \mbV:\ \rot \btau\in \hQ\}.} 
\end{equation}

\begin{assumption}\label{assump}
We assume that $\rot \mbV_h \subset \hQ$ and the existence of a projection 
$\bPi_{V}:\hbV \to \mbV_h$ such that
\begin{subequations}
\label{eqn:Assumpt}
\begin{alignat}{2}
\rot \bPi_V \btau&=  \rot \btau \quad && \forall  \btau \in \hbV, \\
\|\bPi_V \btau-\btau\|_{L^2(\Omega)} &\le   \omega_0(h)  (\|\btau\|_{H^{\frac{1}{2}+\delta}(\Omega)}+ \|\rot \btau\|_{L^2(\Omega)}) \quad && \forall \btau \in \hbV.
\end{alignat}
\end{subequations}
Furthermore, we assume that the $L^2$-orthogonal projection $\Pi_Q:L^2(\Omega)\to \hQ$ satisfies 
\begin{equation*}
\|\Pi_Q \phi-\phi\|_{L^2(\Omega)} \le \omega_1(h) \|\bcurl \phi\|_{L^2(\Omega)} \qquad \forall \phi\in H^1(\Omega)\cap L^2_0(\Omega).
\end{equation*}
Here, the constants are assumed to satisfy $\omega_0(h),\omega_1(h)>0$ and $\lim_{h\to 0^+} \omega_i(h)=0$ for $i=0,1$.
\end{assumption}

\begin{theorem}\label{mainthm}
Suppose that $(\mbV_h, \hQ)$ satisfy Assumption \ref{assump}. 
Let $T$ and $T_h$ be defined by \eqref{source} and \eqref{femsource}, respectively. Then there holds
\begin{equation*}
\|T-T_h\| \le C (\omega_0(h)+ \omega_1(h)).
\end{equation*} 
\end{theorem}
Note that Theorem \ref{mainthm} and Proposition \ref{uniform} imply 
\MJN{that the discrete eigenvalues in the finite element method \eqref{fem}
converge to the correct values.}
To prove Theorem \ref{mainthm}, we require two preliminary results. 
%
\begin{lemma}\label{auxlemma}
Suppose that Assumption \ref{assump} is satisfied.
Then there exists a constant $C>0$ such that
\begin{equation*}
\| \bA\Pi_Q f-\bA f\|_{L^2(\Omega)} + \|T\Pi_Q f-T f\|_{L^2(\Omega)} \le C \omega_1(h) \|f\|_{L^2(\Omega)} \quad \forall f \in L^2(\Omega). 
\end{equation*}
\end{lemma}
\begin{proof}
Let $f \in L^2(\Omega)$  and set $\bsigma=\bA f, u=Tf$, $\bpsi=\bA \Pi_Q f$ and $w= T \Pi_Q f$.
We see that
\begin{subequations}
\label{eqn:auxline}
 \begin{alignat}{2}
 \label{eqn:auxline1}
(\bsigma-\bpsi, \btau)+(u-w, \rot \btau)&=0 \qquad && \forall \btau \in  \bH_0(\rot, \Omega), \\
\label{eqn:auxline2}
(\rot (\bsigma-\bpsi), v)&=(f-\Pi_Q f, v) \quad && \forall  v \in L^2_0(\Omega).
\end{alignat}
\end{subequations} 
Setting $v = w-u$ in \eqref{eqn:auxline2} and $\btau = \bsigma-\bpsi$ in \eqref{eqn:auxline1}, and adding the result yields
$\|\bsigma - \bpsi\|_{L^2(\Omega)}^2 = (f-\Pi_Q f,w-u)$.
Furthermore, \eqref{eqn:auxline1} implies $\bcurl (u-w)=\bsigma-\bpsi$. 
Therefore, there holds
 \begin{equation*}
 \|\bsigma-\bpsi\|_{L^2(\Omega)} \le \sup_{\phi \in H^1(\Omega) \cap L_0^2(\Omega)} \frac{(f-\Pi_Q f, \phi)}{\|\bcurl \phi\|_{L^2(\Omega)}}. 
 \end{equation*}
 However, the properties of the $L^2$ projection
 and Assumption \ref{assump} give us
\begin{equation*}
  \sup_{\phi \in H^1(\Omega) \cap L_0^2(\Omega)} \frac{(f-\Pi_Q f, \phi)}{\|\bcurl \phi\|_{L^2(\Omega)}} = \sup_{\phi \in  H^1(\Omega) \cap L_0^2(\Omega) } \frac{(f, \phi-\Pi_Q \phi)}{\|\bcurl \phi\|_{L^2(\Omega)}} \le  \omega_1(h)\|f\|_{L^2(\Omega)}.
 \end{equation*}
Thus, we have shown
 \begin{equation*}
 \|\bA \Pi_Q f-\bA f \|_{L^2(\Omega)} \le   \omega_1(h) \|f\|_{L^2(\Omega)}.
 \end{equation*}
 \rev{Finally, because $Tf\in L^2_0(\Omega)$, we have by the Poincare inequality}
\begin{equation*}
 \|T \Pi_Q f -T f\|_{L^2(\Omega)} \le C \rev{\|\bcurl (T\Pi_Q \revj{f}-Tf)\|_{L^2(\Omega)} =} \|\bA \Pi_Q f -\bA f \|_{L^2(\Omega)} \le C \omega_1(h) \|f\|_{L^2(\Omega)}.
\end{equation*}
\end{proof}

Next we prove that Assumption \ref{assump} implies
the inf-sup condition for the pair $(\mbV_h,\hQ)$.
\begin{lemma}\label{infsup}
Suppose that Assumption \ref{assump} is satisfied.
Then there exists a constant $C>0$ such that for every  $u_h \in \hQ$, 
there exists $\btau_h \in \mathring{\bV}_h$ such that $\rot \btau_h=u_h$ and $\|\btau_h\|_{L^2(\Omega)} \le C \|u_h\|_{L^2(\Omega)}$.
\end{lemma}
\begin{proof}
Let $\btau \in \bH_0^1(\Omega)$ with $\rot \btau=u_h$ such that $\|\btau\|_{H^1(\Omega)} \le C \|u_h\|_{L^2(\Omega)}$.  
Noting that $\btau \in \hbV$, we define $\btau_h=\bPi_V \btau$ so that $\rot \btau_h=\rot \btau=u_h$.  Moreover, 
\begin{equation*}
\|\btau_h\|_{L^2(\Omega)} \le  C (\|\btau\|_{H^{\frac{1}{2}+\delta}(\Omega)}+ \|\rot \btau\|_{L^2(\Omega)}) \le C \|\btau\|_{{H^1}(\Omega)} \le C \|u_h\|_{L^2(\Omega)}.
\end{equation*}
\end{proof}

Now we can prove Theorem \ref{mainthm}.
\begin{proof}[Proof of Theorem \ref{mainthm}]
Let $f \in L^2(\Omega)$, and set $\bsigma=\bA f, u=Tf$ and  $\bsigma_h=\bA_h f, u_h=T_h f$.  
Let $\bpsi=\bA \Pi_Q f$ and $w= T \Pi_Q f$. 
 
We first derive an estimate for $\bpsi-\bsigma_h$.  Using the inclusion $\rot \mbV_h\subset \hQ$, we see that   
\begin{alignat*}{2}
(\bPi_V\bpsi-\bsigma_h, \btau_h)+(\Pi_Q w -u_h, \rot \btau_h)&= (\bPi_V \bpsi-\bpsi, \btau_h) \qquad && \forall \rev{\btau_h} \in \mbV_h, \\
(\rot (\bPi_V \bpsi-\bsigma_h), v_h)&=0 \quad  &&  \forall  v_h \in \hQ.
\end{alignat*}
Setting $\btau_h=\bPi_V\bpsi-\bsigma_h$ and applying the Cauchy--Schwarz inequality
yields
\begin{equation*}
\|\bPi_V \bpsi-\bsigma_h\|_{L^2(\Omega)} \le \|\bPi_V \bpsi-\bpsi\|_{L^2(\Omega)} \le \omega_0(h) (\|\bpsi\|_{H^{\frac{1}{2}+\delta}(\Omega)}+ \|\rot \bpsi\|_{L^2(\Omega)}).
\end{equation*}
If we use  Proposition \ref{embedding} we get
\begin{equation*}
\|\bpsi\|_{H^{\frac{1}{2}+\delta}(\Omega)}\le C (\|\dive \bpsi\|_{L^2(\Omega)}+ \|\rot \bpsi\|_{L^2(\Omega)})= C  \|\rot \bpsi\|_{L^2(\Omega)} =C \|\Pi_Q f\|_{L^2(\Omega)} \le C  \|f\|_{L^2(\Omega)}.
\end{equation*}
Hence, 
\begin{equation*}
\|\bPi_V\bpsi-\bsigma_h\|_{L^2(\Omega)} \le  C \omega_0(h) \|f\|_{L^2(\Omega)}.
\end{equation*}
Next, we note that by Lemma \ref{auxlemma},
\begin{equation*}
\| \bsigma-\bpsi\|_{L^2(\Omega)} + \|w-u\|_{L^2(\Omega)} \le C \omega_1(h) \|f\|_{L^2(\Omega)},
\end{equation*}
and therefore,
\begin{align*}
\|(\bA -\bA_h) f\|_{L^2(\Omega)} &=\| \bsigma-\bsigma_h\|_{L^2(\Omega)} \\
&\le \|\bsigma - \bpsi\|_{L^2(\Omega)} +\|\bsigma_h - \bPi_V \bpsi\|_{L^2(\Omega)}+\|\bPi_V \bpsi-\bpsi\|_{L^2(\Omega)}\\ 
 &\le C (\omega_0(h)+\omega_1(h)) \|f\|_{L^2(\Omega)}.
\end{align*}

Using the inf-sup stability stated in Lemma \ref{infsup}, we have 
\begin{equation*}
\|\Pi_Q w-u_h\|_{L^2(\Omega)} \le  C \|\bpsi-\bsigma_h\|_{L^2(\Omega)} \le  C( \omega_0(h)+\omega_1(h)) \|f\|_{L^2(\Omega)}.
\end{equation*}
Hence, we have 
\begin{align*}
\|w-u_h\|_{L^2(\Omega)} 
&\le   C (\omega_0(h)+\omega_1(h)) \|f\|_{L^2(\Omega)}+ \|w-\Pi_Q w\|_{L^2(\Omega)}\\
& \le C  (\omega_0(h)+\omega_1(h)) \|f\|_{L^2(\Omega)} +\omega_1(h) \|\bcurl w\|_{L^2(\Omega)}.
\end{align*}
But we have $\|\bcurl w\|_{L^2(\Omega)} \le C \|\Pi_Q f\|_{L^2(\Omega)}\le C \|f\|_{L^2(\Omega)}$, and so
\begin{equation*}
\|(T-T_h) f\|_{L^2(\Omega)} =\|u-u_h\|_{L^2(\Omega)} \le  C( \omega_0(h)+\omega_1(h))  \|f\|_{L^2(\Omega)}. 
\end{equation*}

\end{proof}

\section{Examples of Fortin Operators}\label{sec-Examples}
In this section we give examples of finite element pairs satisfying Assumption \ref{assump},
where $\mbV_h$ is taken to be a space of continuous, piecewise polynomials, i.e.,
a Lagrange finite element space.  Here we use recent results on divergence-free finite element
pairs for the Stokes problem to construct a Fortin projection satisfying 
\eqref{eqn:Assumpt}.  A common theme of these Stokes pairs is the imposition of mesh conditions
for low-polynomial degree finite element spaces; it is well-known that
Assumption \ref{assump} is not satisfied on general simplicial meshes and for low polynomial degree.  
Before continuing, we introduce
some notation.

We denote by $\mct$ a shape-regular, simplicial triangulation of $\Omega$
with $h_T = {\rm diam}(T)$ for all $T\in \mct$, and $h = \max_{T\in \mct} h_T$.
Let $\calV^I_h$, $\calV_h^B$, $\calV_h^C$ denote the sets of interior vertices,
boundary vertices, and corner vertices, respectively.
\MJN{Note that the cardinality of $\calV_h^C$ is uniformly bounded
due to the shape-regularity of $\mct$.}
The set of all vertices is $\calV_h = \calV_h^I\cup \calV_h^B$.
Likewise, $\calE_h^I$ and $\calE_h^B$ are the sets of interior
and boundary edges, respectively, and $\calE_h = \calE_h^I\cup \calE_h^B$.
We denote by $\mct(z)$ the patch of triangles that have $z\in \calV_h$ as a vertex.
Likewise, $\calV_h^I(T)$ and $\calV_h^B(T)$ are the sets
of interior and boundary vertices of $T\in \mct$, 
and $\calE_h^I(T)$ is the set of interior edges of $T$.

For a non-negative integer $k$ and  set $S\subset \Omega$,
 let $\pol_k(S)$ to be the space of piecewise polynomials
of degree $\le k$ with domain $S$.  The analogous space of piecewise 
polynomials with respect to $\mct$ is
\[
\pol_k(\mct) = \prod_{T\in \mct} \pol_k(T),
\]
and the Lagrange finite element space is
\[
\pol_k^c(\mct) = \pol_k(\mct)\cap H^1(\Omega).
\]
Analogous vector-valued spaces are denoted in boldface, e.g., $\bpol_k(\mct) = [\pol_k(\mct)]^2$.
Finally, the constant $C$ denotes a generic constant
that is independent of the mesh parameter $h$ and may take
different values at different occurrences.

In the subsequent sections, 
we will employ a Scott--Zhang type interpolant
on the space $\mbV$. 
We cannot use the Scott--Zhang interpolant directly, as the canonical Scott--Zhang interpolant of a 
function in $\mbV$ might not have zero tangential components at the corners of  $\Omega$; 
hence, we have to modify the Scott--Zhang interpolant at the corners of $\Omega$. \revj{This type of interpolant has been used for example in \cite[(2.14) and (2.15)]{BonitoGuermond11}.
For completeness we give a detailed construction in the appendix but we state the result here.}
%
\begin{lemma}\label{lem:SZ}
Let $0< \delta \le \frac{1}{2}$.  There exists a projection $\bI_h: \bH^{\frac{1}{2}+\delta}(\Omega)\cap \bH_0({\rm rot},\Omega) \rightarrow  \bpol_1^c(\mct)\cap \bH_0({\rm rot},\Omega)$ with the following bound:
\begin{equation}\label{eqn:SZmod}
h_T^{-\frac{1}{2}-\delta} \|\btau-\bI_h \btau\|_{L^2(T)} + \|\bI_h \btau\|_{H^{\frac{1}{2}+\delta}(T)} \le C\|\btau\|_{H^{\frac{1}{2}+\delta}(\omega(T))}\qquad \forall \btau\in \mbV,
\end{equation}
where $\omega(T) = \displaystyle\mathop{\bigcup_{T'\in \mct}}_{\bar T\cap \bar T'\neq \emptyset} T'$.
\end{lemma}

\subsection{Construction of Fortin Operator on Powell--Sabin Splits}\label{sec-PS}
In this section, we use the recent results
given in \cite{GLN19} to construct a Fortin projection
into the Lagrange finite element space
defined on Powell--Sabin triangulations.
For simplicity and readability, we focus on the lowest-order case;
however, the arguments easily extend to arbitrary polynomial degree $k\ge 1$.

Given the simplicial triangulation of $\mct$ of $\Omega$,
we construct its Powell--Sabin refinement $\PST$
as follows \cite{PowellSabin77,SplineBook,GLN19}: 
First, adjoin the incenter of each $T\in \mct$ to each
vertex of $T$. Next, the interior points (incenters) of each adjacent pair of triangles are connected with an edge.
For any $T$ that shares an edge with the boundary of $\Omega$, the midpoint
of that edge is connected with the incenter of $T$.  
Thus, each $T\in \mct$ is split into six triangles; cf.~Figure \ref{fig:PS}.
\begin{figure}
\begin{center}
\includegraphics[scale=0.25]{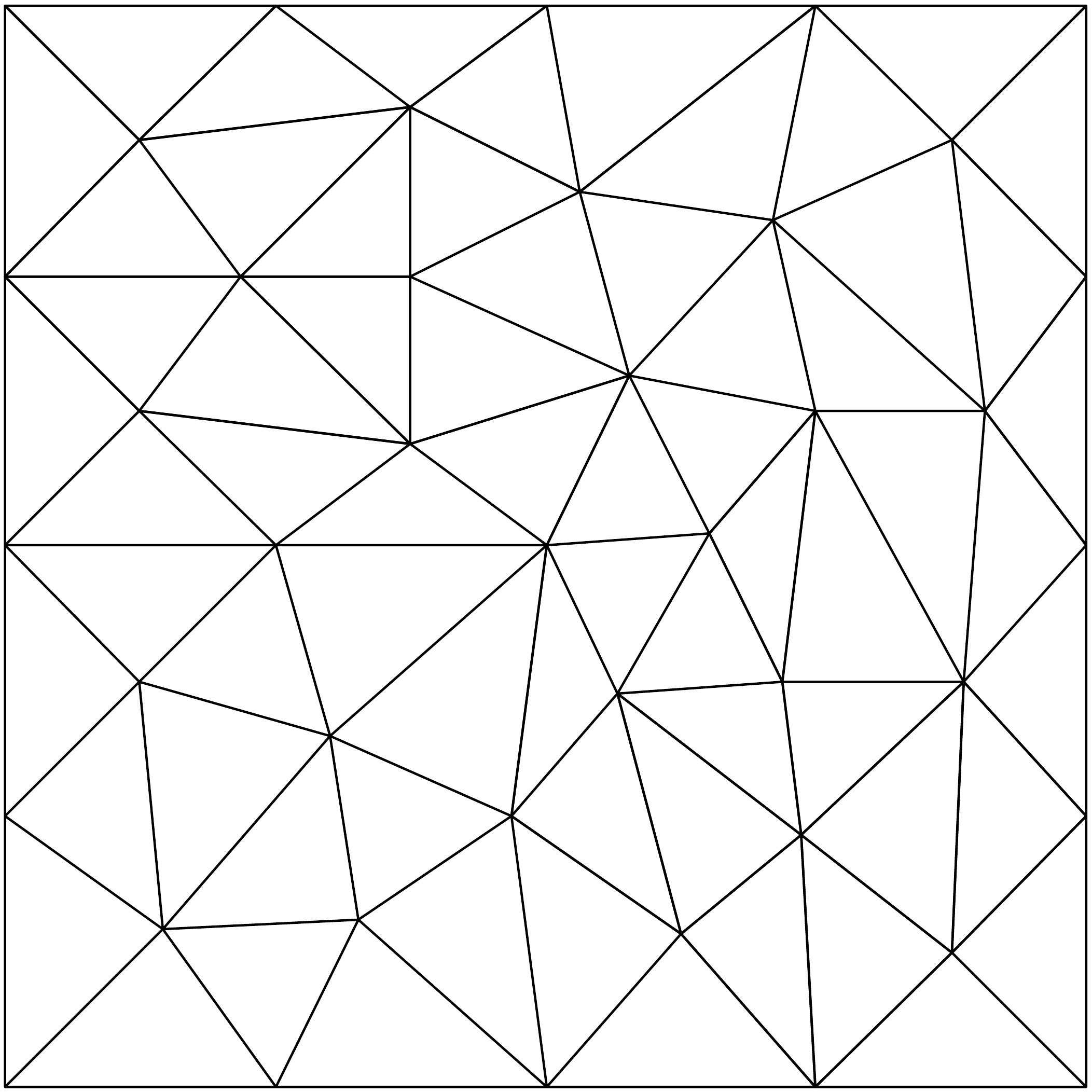}\quad
\includegraphics[scale=0.25]{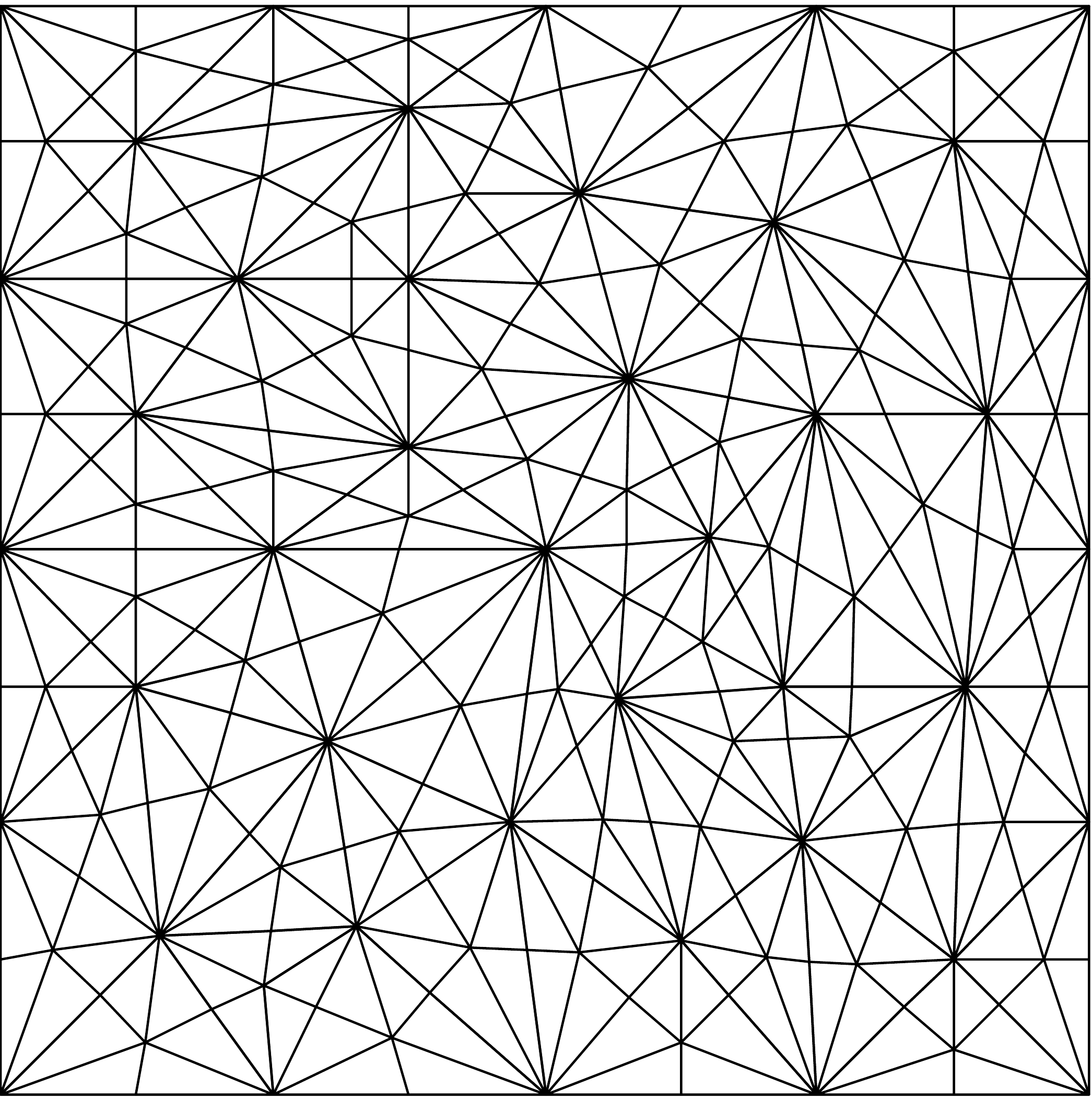}
\caption{A simplicial triangulation of the unit square (left)
and the associated Powell--Sabin triangulation (right).}
\label{fig:PS}
\end{center}
\end{figure}

Let $\calS^I_h(\PST)$ be the points of intersection
of the {interior} edges of $\mct$ that adjoin incenters,
 let $\calS^B_h(\PST)$ be the intersection points of 
the boundary edges that adjoin incenters, and set $\calS_h(\PST) = \calS_h^I(\PST)\cup \calS_h^B(\PST)$.
Note that, by the definition of the Powell--Sabin split,
the points in $\calS_h(\PST)$ are the singular vertices in $\PST$, i.e,
the vertices that lie on exactly two straight lines.
In particular, for a vertex $z\in \calS^I_h(\PST)$
there exists four triangles $\PST(z) = \{T_i\}_{i=1}^4\subset \PST$
such that $z$ is a vertex of $T_i$.  Without loss of generality
we assume that these triangles are labeled in a counterclockwise direction.
We then define
for a scalar function $v$,
\begin{equation}\label{eqn:SAlt}
\theta_z(v):= v|_{T_1}(z) - v|_{T_2}(z) + v|_{T_3}(z) - v|_{T_4}(z).
\end{equation}

We then define the spaces
\begin{subequations}
\label{eqn:PSSpaces}
\begin{align}
%
\mbV_h & =\bpol^c_1(\PST)\cap \bH_0({\rm rot},\Omega),\\
\hQ
& = \{v\in \pol_0(\PST) \cap L^2_0(\Omega):\ \theta_z(v) = 0\ \forall z\in \calS_h^I(\PST)\}.
\end{align}
\end{subequations}
%

\begin{lemma}[\cite{GLN19}]\label{thm:DExact}
Let $\mbV_h$ and $\hQ$ be defined by \eqref{eqn:PSSpaces}.
Then there holds ${\rm rot}\,\mbV_h \subset \hQ$.
\end{lemma}

We now extend the results of
\cite{GLN19} to construct an appropriate Fortin operator
that is well defined for $\btau \in \hbV$.
To do so, 
we require some additional notation.

For an interior singular vertex $z\in \calS_h(\PST)$, 
let $T\in \mct$ be a triangle in $\mct$
such that $z\in \p T$, and let $\{K_1,K_2\}\subset \PST$
be the triangles in $\PST$ such that $K_1,K_2\subset T$
and $K_1,K_2\in \PST(z)$.  Let $e = \p K_1\cap \p K_2$,
and let ${\bm m}_i$ be the outward unit normal of $K_i$
perpendicular to $e$.  We then define the jump
of a scalar piecewise smooth function at $z$ (restricted to $T$) as 
\[
\jump{v}_T(z) = v\big|_{K_1}(z){\bm m}_1 + v\big|_{K_2}(z) {\bm m}_2.
\]
Note that $\jump{v}_T(z)$ is single-valued for all $v\in \hQ$.
\MJN{In particular, if $z$ is an interior singular vertex
with $z\in \p T_1\cap \p T_2$ for some $T_1,T_2\in \mct$, $T_1\neq T_2$,
then $\jump{v}_{T_1}(z) = \jump{v}_{T_2}(z)$ for all $v\in \hQ$ 
because $\theta_z(v)=0$.  Therefore, we shall omit the subscript
and simply write $\jump{v}(z)$.}

Next for a triangle $T\in \mct$ in the non-refined mesh,
we denote by $T^{\rm ct}$ the resulting set
of three triangles obtained by connecting the barycenter
of $T$ to its vertices, i.e., $T^{\rm ct}$ is the Clough--Tocher
refinement of $T$.  We define the set of (local) piecewise polynomials
with respect to this partition as
\begin{equation}\label{eqn:LocalCT}
\pol_k(T^{\rm ct}) = \prod_{K\in T^{\rm ct}} \pol_k(K).
\end{equation}

The following lemma provides
the degrees of freedom for $\mbV_h$ and $\hQ$ that will be
used to construct the Fortin operator.
The result essentially follows from 
\cite[\rev{Lemmas 10--11}]{GLN19}.
\begin{lemma}
A function $\btau\in \mbV_h$ is uniquely defined by the conditions
\begin{subequations}
\label{eqn:VhDOFs}
\begin{alignat}{3}
\label{eqn:VhDOFs1}
& \btau(z)\qquad &&\forall z\in \calV_h^I,\\
\label{eqn:VhDOFs2}
& \btau(z)\cdot \bn\qquad &&\forall z\in \calV_h^B\backslash \calV_h^C,\\
\label{eqn:VhDOFs3}
& \int_e (\btau\cdot \bt)\qquad &&\forall e\in \calE_h^I,\\
&\jump{\rot \btau}(z)\qquad &&\forall z\in \calS_h(\PST),\\
%
%
& \int_T (\rot \btau) r \qquad &&\forall r\in {\pol}_0(T^{\rm ct})\cap L^2_0(T),\ \forall T\in \mct.
\end{alignat}
\end{subequations}
Moreover, a function $v\in \hQ$ is uniquely determined
by the values
\begin{subequations}
\label{eqn:QDOFs}
\begin{alignat}{2}
&\jump{v}(z)\qquad &&\forall z\in \calS_h(\PST),\\
&\int_T v r\qquad &&\forall r\in \mathring{\pol}_0(T^{\rm ct}),\ \forall T\in \mct.
\end{alignat}
\end{subequations}
\end{lemma}
%

\begin{theorem}\label{thm:PSMainThm}
Let $\mbV_h$ and $\hQ$ be defined by \eqref{eqn:PSSpaces},
and let $\hbV$ be defined
by \eqref{eqn:FunkySpace}.
Then there exists a projection $\bPi_V:\hbV \to \mbV_h$
such that $\rot \bPi_V \bp = \rot \bp$ for all $\bp\in \hbV$.  Moreover,
\[
\|\btau-\bPi_V \btau\|_{L^2(\Omega)}\le C\big( h^{\frac{1}{2}+\delta} \|\btau\|_{H^{\frac{1}{2}+\delta}(\Omega)} + h\|\rot \btau\|_{L^2(\Omega)}\big).
\]
\end{theorem}
\begin{proof}
Fix $\btau\in \hbV$,
and let $\bI_h \btau \in \bpol_1^c(\mct)\cap \bH_0({\rm rot},\Omega)\subset \mbV_h$ be 
the modified Scott--Zhang interpolant of $\btau$ established
in Lemma \ref{lem:SZ}.
We then construct $\bPi_V \btau$ via
the conditions
\begin{subequations}
\label{eqn:FortinDef}
\begin{alignat}{3}
& (\bPi_V \btau)(z) = (\bI_h \btau)(z)\qquad &&\forall z\in \calV_h^I,\\
& (\bPi_V \btau)(z)\cdot \bn = (\bI_h \btau)(z)\cdot \bn\qquad &&\forall z\in \calV_h^B\backslash \calV_h^C,\\
& \int_e (\bPi_V \btau)\cdot \bt = \int_e \btau\cdot \bt \qquad &&\forall e\in \calE_h^I,\\
&\jump{\rot \bPi_V \btau}(z) = \jump{\rot \btau}(z)\qquad &&\forall z\in \calS_h(\PST),\\
& \int_T (\rot \bPi_V \btau) \revj{r} = \int_T (\rot \btau) \revj{r} \qquad &&\forall r\in \rev{{\pol}_0(T^{\rm ct})\cap L^2_0(T)},\ \forall T\in \mct.
\end{alignat}
\end{subequations}
The arguments given in \cite{GLN19}
show that $\rot \bPi_V \btau = \rot \btau$,

By scaling, there holds for each $\bsigma_h\in \mbV_h$ and on each $T\in \mct$,
\begin{align*}
\|\bsigma_h\|_{L^2(T)}^2 
&\le C\Big[  h^{2}_T  \Big(\sum_{z\in \calV_h^I(T)} |\bsigma_h(z)|^2 +\sum_{z\in \calV_h^B(T)\backslash \calV_h^C(T)} |\bsigma_h(z)\cdot \bn|^2\Big)\\
&\qquad + 
 \sum_{e\in \calE_h^I(T)}\Big|\int_e \bsigma_h\cdot \bt\Big|^2 +  h^{2}_T \mathop{\sup_{r\in \mathcal{\pol}_0(T^{\rm ct})}}_{\|r\|_{L^2(T)}=1} \Big| \int_T (\rot \bsigma_h)r\Big|^2\\
&\qquad+h_T^{4}\sum_{z\in \calS_h(T)} |\jump{\rot \bsigma_h}(z)|^2\Big],
\end{align*}
\rev{where $\calS_h(T)$ is set of singular vertices contained in $\bar T$.}
Now set $\bsigma_h = \bPi_V \btau - \bI_h \btau$.  Using the above estimate and \eqref{eqn:FortinDef} 
then yields
\begin{align}\label{eqn:Step1}
\|\bPi_V \btau - \bI_h \btau\|_{L^2(T)}^2 
&\le C\Big[  \Big|\int_{\p T} (\btau-\bI_h \btau)\cdot \bt\Big|^2 +  h^{2}_T \mathop{\sup_{r\in \mathcal{\pol}_0(T^{\rm ct})}}_{\|r\|_{L^2(T)}=1} \Big| \int_T (\rot (\btau-\bI_h \btau)r\Big|^2\\
&\nonum\qquad +h_T^{4}\sum_{z\in \calS_h(T)} |\jump{\rot (\btau-\bI_h \btau)}(z)|^2\Big].
\end{align}

Because $\rot (\btau-\bI_h \btau)\in \hQ$, we use the degrees of freedom
 \eqref{eqn:QDOFs} and a scaling argument to conclude that
\begin{align}
\label{eqn:Step2}
& \mathop{\sup_{r\in \mathcal{\pol}_0(T^{\rm ct})}}_{\|r\|_{L^2(T)}=1} \Big| \int_T (\rot (\btau-\bI_h \btau)r\Big|^2
+h_T^{2}\sum_{\rev{z\in \calS_h(T)}} |\jump{\rot (\btau-\bI_h \btau)}(z)|^2\\
&\nonumber\qquad\le C  \|\rot (\btau-\bI_h \btau)\|_{L^{2}(T)}^2.
\end{align}
\MJN{We then use an inverse estimate 
to get
\begin{align}
\label{eqn:Step3}
\|\rot (\btau-\bI_h \btau)\|_{L^{2}(T)}^2
&\le C \big[\|\rot \btau\|_{L^2(T)}^2+ \|\nab \bI_h \btau\|_{L^2(T)}\big]\\
&\nonum\le C\big[ \|\rot \btau\|_{L^2(T)}^2 + h_T^{-1+2\delta}\|\bI_h \btau\|_{H^{\frac{1}{2}+\delta}(T)}^2\big].
%
\end{align}
Applying the estimates \eqref{eqn:Step2}--\eqref{eqn:Step3} to \eqref{eqn:Step1}, we obtain
\begin{align*}
\|\bPi_V \btau - \bI_h \btau\|_{L^2(T)}^2 
&\le C\Big[  \Big|\int_{\p T} (\btau-\bI_h \btau)\cdot \bt\Big|^2 
+   h_T^{2}\big(\|\rot \btau\|_{L^2(T)}^2 + h_T^{-1+2\delta}\| \bI_h \btau\|_{H^{\frac{1}{2}+\delta}(\omega(T))}^2\big)\Big]\\
&\le C\big[ h_T \|\btau-\bI_h \btau\|_{L^2(\p T)}^2 + h_T^{1+2\delta }  \|\bI_h \btau\|_{H^{\frac{1}{2}+\delta}(T)}^2
+ h_T^{2} \|\rot \btau\|_{L^2(T)}^2\big].
%
%
\end{align*}
We then apply \eqref{eqn:SZmod} and sum over $T\in \mct$ to obtain
\begin{align*}
\rev{\|\bPi_V \btau - \bI_h \btau\|_{L^2(\Omega)}}
& \le C\big[ h^{\frac{1}{2}+\delta} \|\btau\|_{H^{\frac{1}{2}+\delta}(\Omega)}+h \|\rot \btau\|_{L^2(\Omega)}\big].
\end{align*}
}
Therefore
\begin{align*}
\|\btau - \bPi_V \btau\|_{L^2(\Omega)}
&\le \|\btau-\bI_h \btau\|_{L^2(\Omega)} +\|\bPi_V \btau - \bI_h \btau\|_{L^2(\Omega)}\\
&\le C\big[ h^{\frac{1}{2}+\delta} \|\btau\|_{H^{\frac{1}{2}+\delta}(\Omega)}
+ h \|\rot \btau\|_{L^2(\Omega)}\big].
\end{align*}
\end{proof}

\subsection{Construction of Fortin Operator on Clough--Tocher Splits}\label{sec-CT}
The Clough--Tocher refinement of $\mct$ is obtained
by connecting the barycenter of each $T\in \mct$ with its vertices;
thus, each triangle is split into three triangles.  
In this section, we show that there exists a Fortin projection
mapping onto the Lagrange finite element space satisfying Assumption \ref{assump}.
This result holds for all polynomial degrees $k\ge 2$ but, for simplicity,
we only consider the lowest order case $k=2$.

Let $\CTT$ be the resulting Clough--Tocher refinement of $\mct$,
and define the spaces
\begin{subequations}
\label{eqn:CTSpaces}
\begin{align}
%
\mbV_h
& = \bpol^c_2(\CTT)\cap \bH_0({\rm rot},\Omega),\\
%
\hQ & = L^2_0(\Omega)\cap \pol_1(\CTT).
\end{align}
\end{subequations}
It is well-known that ${\rm rot}\,\mbV_h \subset \hQ$ \cite{FGN20}.

Below we modify the results in \cite{FGN20} to build a Fortin projection that is well-defined on $\bH^{\frac{1}{2}+\delta}(\Omega)$
and has optimal order convergence properties in $\bL^2(\Omega)$.
To this end, we first provide a useful set of degrees of freedom for $\mbV_h$ \cite{FGN20}.

\begin{lemma}
A function $\btau\in \mbV_h$ is uniquely determined by the values
\begin{alignat}{2}
&\btau(z)\qquad &&\forall z\in \calV_h^I,\\
&\btau(z)\cdot \bn \qquad &&\forall z\in \calV_h^B\backslash \calV_h^C,\\
&\int_e \btau \qquad &&\forall e\in \calE_h^I,\\
&\int_e \btau\cdot \bn\qquad &&\forall e\in \calE_h^B,\\
&\int_T ({\rm rot}\,\btau)r\qquad &&\forall r\in \pol_1(T^{\rm ct})\cap L^2_0(T),\ \forall T\in \mct,
\end{alignat}
where $\pol_1(T^{\rm ct})$ is defined by \eqref{eqn:LocalCT}.
\end{lemma}
%

\begin{theorem}\label{thm:CTmain}
Let $\mbV_h$ and $\hQ$ be defined by \eqref{eqn:CTSpaces},
and let $\Pi_Q$ be the $L^2$ projection onto $\hQ$.
Then there exists a projection $\bPi_V: \hbV \to \mbV_h$,
such that ${\rm rot}\,\bPi_V \btau = \Pi_Q ({\rm rot}\,\btau)$,
 Moreover,
\[
\|\btau-\bPi_V \btau\|_{L^2(\Omega)}\le C\big( h^{\frac{1}{2}+\delta} \|\btau\|_{H^{\frac{1}{2}+\delta}(\Omega)} + h\|{\rm rot}\,\btau\|_{L^2(\Omega)}\big).
\]
\end{theorem}
\begin{proof}

Define $\bPi_V$ uniquely by the conditions
\begin{subequations}
\label{eqn:CTFortinDef}
\begin{alignat}{3}
& (\bPi_V \btau)(z) = (\bI_h \btau)(z)\qquad &&\forall z\in \calV_h^I,\\
& (\bPi_V \btau)(z)\cdot \bn = (\bI_h \btau)(z)\cdot \bn\qquad &&\forall z\in \calV_h^B\backslash \calV_h^C,\\
& \int_e (\bPi_V \btau) = \int_e \btau \qquad &&\forall e\in \calE_h^I,\\
&\int_e (\bPi_V \btau\cdot \bn) = \int_e \btau \cdot \bn \qquad &&\forall e\in \calE_h^B,\\
&\int_T ({\rm rot}\,\bPi_V \btau)r = \int_T ({\rm rot}\, \btau) r \qquad &&\forall r\in \pol_1(T^{\rm ct})\cap L^2_0(T),\ \forall T\in \mct.
%
%
\end{alignat}
\end{subequations}
The arguments given in \cite{FGN20}
show that $\rot \bPi_V \btau = \Pi_Q \rot \btau$.
The same scaling arguments given in Theorem \ref{thm:PSMainThm} show that
 $\|\btau- \bPi_V \btau\|_{L^2(\Omega)}\le C\big( h^{\frac{1}{2}+\delta} \|\btau\|_{H^{\frac{1}{2}+\delta}(\Omega)} + h\|{\rm rot}\,\btau\|_{L^2(\Omega)}\big)$.
\end{proof}

\subsection{Construction of Fortin Operator on General Triangulations}\label{sec-SV}

In this section, we construct
a Fortin operator for the original Scott--Vogelius pair
developed in \cite{ScottVogelius85}.
This pair essentially takes the space $\mbV_h$ to be
the Lagrange space of degree $k\ge 4$,
and $\hQ$ to be the space of piecewise polynomials
of degree $(k-1)$.   As pointed out in \cite{ScottVogelius85}
the exact definition of these spaces and their stability
is mesh-dependent and depends on the presence
of singular or ``nearly singular'' vertices.  

Recall
that a singular vertex is a vertex in $\mct$ 
that lies on exactly two straight lines.  To make this precise,
for a vertex $z\in \calV_h$, we enumerate the triangles
that have $z$ as a vertex as
 $\mct(z)=\{ T_1, T_2, \ldots T_N \}$. 
If $z$ is a boundary vertex then we enumerate the triangles such that $T_1$ and  $T_N$ 
have a boundary edge.  Moreover, we enumerate them so that $T_j, T_{j+1}$ share an 
edge for $j=1, \ldots N-1$ and  $T_N$ and $T_1$ share an edge in the case $z$ is an interior vertex. Let $\theta_j $ denote the angle between the edges of $T_j$ originating from $z$.   We define 
\begin{equation}
\label{eqn:ThetaDef}
\Theta(z)= 
\begin{cases}
\max \{ |\sin(\theta_1+\theta_{2})|,  \ldots, |\sin(\theta_{N-1}+\theta_{N})|, |\sin(\theta_N+\theta_1)|\} & \text{ if }  z\in \calV_h^I\\
\max \{ |\sin(\theta_1+\theta_{2})|,   \ldots, |\sin(\theta_{N-1}+\theta_{N})| \} & \text{ if }   z\in \calV_h^B\text{ and }N\ge 2,\\
0 & \text{ if } z\in \calV_h^B\text{ and } N=1.
\end{cases}
\end{equation}

\begin{definition}
A vertex $z \in \calV_h$ is a singular vertex if $\Theta(z)=0$. It is non-singular if $\Theta(z) >0$.
\end{definition}

We denote all the singular vertices by 
\begin{equation*}
\Sh=\{ z \in \calV_h:\ \Theta(z)=0\}.
\end{equation*}
We further let
$\Sh^I$ denote the set of interior singular vertices, $\Sh^B$  the set of boundary singular vertices,
and $\Sh^C$  the set of corner singular vertices.  Equivalently,
\begin{align*}
\Sh^I &= \{z\in \Sh:\ \# \mct(z) = 4\},\\
\Sh^B & = \{z\in \Sh:\ \# \mct(z) \in \{1,2\}\},\\
\Sh^C & = \{z\in \Sh:\ \# \mct(z) =1\}.
\end{align*}

\begin{definition}
We set 
\begin{equation}\label{eqn:Thetamin}
\Theta_{\min}:= \min_{z\in \calV_h\backslash \calS_h} \Theta(z).
\end{equation}
\end{definition}

For a non-negative integer $k$, we define the spaces
\begin{subequations}
\label{eqn:SVSpaces}
\begin{align}
%
\mbV_h & = \bpol_k^c(\mct)\cap \bH_0({\rm rot},\Omega),\\
%
\hQ & = \{v\in L^2_0(\Omega)\cap \pol_{k-1}(\mct):\ \theta_z(v) = 0\ \forall z\in \calS_h^I,\ v(z)=0\ \forall z\in \calS_h^C\},
\end{align}
\end{subequations}
where we recall that $\theta_z(v)$ is defined by \eqref{eqn:SAlt}.

First we note that the rot operator maps $\mbV_h$ into $\hQ$ \cite{ScottVogelius85}.
\begin{lemma}
There holds ${\rm rot}\,\btau\in \hQ$ for all $\btau\in \mbV_h$.
\end{lemma}

Let $\bI_h$ be  Scott--Zhang interpolant onto $\bpol^c_1(\mct)\cap \bH_0({\rm rot};\Omega)\subset \mbV_h$. 
Then define \\
 $\bI_1:\bH^{\frac{1}{2}+\delta}(\Omega)\to \mbV_h$ as follows
\begin{alignat*}{2}
\bI_1\btau (z)&=\bI_h \btau(z) \quad && \forall z\in \calV_h,\\
\int_{e} \bI_1 \btau \cdot \bpsi &= \int_{e} \btau  \cdot \bpsi  \quad && \text{ for all } \bpsi \in \bpol_{k-2}(e), \ \forall e\in \calE_h, \\
\int_{T} \bI_1 \btau \cdot \bpsi &= \int_{T} \btau \cdot \bpsi   \quad && \text{ for all } \bpsi \in \bpol_{k-3}(T),\ \forall T\in \mct.
\end{alignat*}
Standard arguments yield the following result.
\begin{lemma}\label{lem:LagrangeInterp}
There holds for all $\btau\in \bH^{\frac{1}{2}+\delta}(\Omega)$
\begin{equation}\label{lem:LagrangeInterp1}
\|\btau-\bI_1 \btau \|_{L^2(\Omega)} \le C h^{\frac{1}{2}+ \delta} \| \btau\|_{H^{\frac{1}{2}+\delta}(\Omega)}. 
\end{equation}
and
\begin{equation} \label{lem:LagrangeInterp2}
\|{\rm rot}(\bI_1 \btau)\|_{L^2(\Omega)} \le  h^{-\frac{1}{2}+\delta}\|\btau\|_{H^{\frac{1}{2}+\delta}(\Omega)}. 
\end{equation}
Moreover, for $k\ge 2$,
\begin{equation}\label{lem:LagrangeInterp3}
\int_{T} {\rm rot}\, \bI_1 \btau = \int_T {\rm rot}\, \btau\quad \forall T\in \mct.
\end{equation}
\end{lemma}

The following result  follows from \cite[Lemma 6]{GuzmanScott19}.
\begin{lemma}\label{lem:SurjectiveVertices}
Suppose that $k\ge 4$.
Then there exists an injective linear operator $\bJ_1:\hQ\to  \mbV_h$ such that
\begin{subequations}
\begin{alignat}{2}
{\rm rot} ( \bJ_1 v) (z)&= v\big (z) \quad && \ \forall z\in \calV_h,\\
\int_{T} {\rm rot} ( \bJ_1 v) \, dx &= 0 \quad && \forall T\in \mct,\\
\| \bJ_1 v\|_{L^2(\Omega)} + h \| \nabla  \bJ_1 v\|_{L^2(\Omega)} 
&\le
C h\Big(\frac1{\Theta_{\min}}+1\Big)&& \|v\|_{L^2(\Omega)}.
\end{alignat}
\end{subequations}
\end{lemma}

%
The following result follows from \cite{ScottVogelius85,FalkNeilan13,GuzmanScott19}.
\begin{lemma}\label{lem:LocalSurjectivity}
Define 
\[
\mathring{\calQ}_h  = \{v\in \hQ:\ \int_T v=0\ \forall T \in \mct, \text{ and }  v(z)=0\ \forall z \in \calV_h\}.
\]
Then there exists an injective operator 
$\bJ_2:\mathring{\calQ}_h\to \mbV_h$ such that 
\begin{alignat*}{1}
{\rm rot} (\bJ_2 v)&= v,  \\
\|\bJ_2 v\|_{L^2(\Omega)}+  h \| \nabla \bJ_2 v\|_{L^2(\Omega)}  &\le C h \|v\|_{L^2(\Omega)} . 
\end{alignat*}
\end{lemma}
%

\begin{theorem}\label{thm:SVmain}
Let $\mbV_h$ and $\hQ$ be defined by \eqref{eqn:SVSpaces} with $k\ge 4$.
Then there exists
a projection $\bPi_{V}:\hbV \to \mbV_h$
such that
\begin{equation*}
{\rm rot} (\bPi_{V} \btau)={\rm rot}\,\btau
\end{equation*}
with the following bound
\begin{equation*}
\|\btau-\bPi_{V} \btau\|_{L^2(\Omega)} \le C\big(1+\Theta_{\min}^{-1}\big) h^{\frac{1}{2}+\delta} \|\btau\|_{H^{\frac{1}{2}+\delta}(\Omega)}.
\end{equation*}
\end{theorem}

%
%

\begin{proof}
Define:
\begin{equation*}
\bPi_{V} \btau = \bI_1 \btau+ \bJ_1v_1 +\bJ_2v_2 \in \mbV_h,
\end{equation*}
where 
\begin{equation*}
v_1= {\rm rot}(\btau-\bI_1 \btau)\in \hQ,\qquad v_2= v_1-{\rm rot}(\bJ_1v_1)\in \hQ.
\end{equation*}

By Lemma \ref{lem:SurjectiveVertices} and the definition of $v_2$, 
we see that 
\begin{align*}
v_2(z)=0 \qquad \forall z\in \calV_h,
\end{align*}
and 
\begin{align*}
\int_T v_2 = \int_T \big(v_1 - {\rm rot} (\bJ_1 v_1)\big) = \int_T v_1 = \int_T {\rm rot} (\btau-\bI_1 \btau) = 0,
\end{align*}
by Lemma \ref{lem:LagrangeInterp}.
Therefore $v_2\in \mathring{\calQ}_h$, and so
 $\bJ_2 v_2$ is well-defined (cf.~Lemma \ref{lem:LocalSurjectivity}).

We then use Lemma \ref{lem:LocalSurjectivity} to get
\begin{alignat*}{1}
{\rm rot} (\bPi_{V} \btau)&={\rm rot}(\bI_1 \btau)+ {\rm rot}(\bJ_1v_1)+ {\rm rot} (\bJ_2v_2) \\
&= {\rm rot}(\bI_1 \btau)+{\rm rot}(\bJ_1 v_1)+  v_2\\
&=  {\rm rot}(\bI_1 \btau)+{\rm rot}(\bJ_1v_1)+  (v_1-{\rm rot} (\bJ_1v_1))\\
&=  {\rm rot}(\bI_1 \btau)+v_1\\
&={\rm rot}(\bI_1 \btau)+{\rm rot} (\btau-\bI_1 \btau)\\
&=  {\rm rot}\,\btau.
\end{alignat*}

Now we note that, by \eqref{lem:LagrangeInterp2},
\begin{alignat}{1}
\|{\rm rot} (\btau-\bI_1 \btau)\|_{L^2(\Omega)} &\le  \|\rot \btau\|_{L^2(\Omega)}  +\|\rot( \bI_1 \btau)\|_{L^2(\Omega)}  \nonumber \\
&\le  \|\rot \btau\|_{L^2(\Omega)}  +h^{-\frac{1}{2}+\delta} \|\btau \|_{H^{\frac{1}{2}+\delta}(\Omega)}. \label{841}
\end{alignat}

Next, by Lemma \ref{lem:SurjectiveVertices} and \eqref{841} we obtain
\begin{align}\label{eqn:J1q1Bound}
\|\bJ_1 v_1\|_{L^2(\Omega)}
&\le Ch\Big(\frac1{\Theta_{\min}} +1\Big)\|v_1\|_{L^2(\Omega)}\\
&\nonumber\le Ch\Big(\frac1{\Theta_{\min}} +1\Big)\|{\rm rot} (\btau-\bI_1 \btau)\|_{L^2(\Omega)}\\
&\nonumber\le  C\Big(\frac1{\Theta_{\min}} +1\Big)  (h\|\rot \btau\|_{L^2(\Omega)}  +h^{\frac{1}{2}+\delta} \|\btau \|_{H^{\frac{1}{2}+\delta}(\Omega)}) .
\end{align}

Likewise, we use Lemmas \ref{lem:LocalSurjectivity} and \eqref{841} to obtain
\begin{align}\label{eqn:J2q2Bound}
\|\bJ_2 v_2\|_{L^2(\Omega)}
&\le C h \|v_2\|_{L^2(\Omega)}\\
&\nonumber\le C h \big(\|v_1\|_{L^2(\Omega)} + \|{\rm rot} (\bJ_1 v_1)\|_{L^2(\Omega)}\big)\\
&\nonumber\le C  \big(h\|{\rm rot} (\btau-\bI_1 \btau)\|_{L^2(\Omega)} +\|\bJ_1 v_1\|_{L^2(\Omega)}\big)\\
& \nonumber \le  C\Big(\frac1{\Theta_{\min}} +1\Big)  (h\|\rot \btau\|_{L^2(\Omega)}  +h^{\frac{1}{2}+\delta} \|\btau \|_{H^{\frac{1}{2}+\delta}(\Omega)}).
\end{align}

We then use the triangle inequality, Lemma \ref{lem:LagrangeInterp}, 
\eqref{eqn:J1q1Bound}, and \eqref{eqn:J2q2Bound} to obtain the $L^2$ error estimate:
\begin{alignat*}{1}
\|\btau-\bPi_{V} \btau\|_{L^2(\Omega)}& \le \|\btau-\bI_1 \btau\|_{L^2(\Omega)} + \|\bJ_1v_1\|_{L^2(\Omega)}+ \|\bJ_2v_2\|_{L^2(\Omega)} \\
&  \le C\Big(\frac1{\Theta_{\min}}+1\Big)  h^{\frac{1}{2}+\delta}\|\btau\|_{H^{\frac{1}{2}+\delta}(\Omega)}.
\end{alignat*}

Finally, if $\btau\in \mbV_h$, then $\bI_1 \btau=\btau$ and so $v_1 = 0$.
It then follows that $\bJ_1 v_1 = 0$, and $\bJ_2 v_2 = -\bJ_2 ({\rm rot} (\bJ_1 v_1)) = 0$.
Therefore $\bPi_{V} \btau = \bI_1  \btau= \btau$, i.e., $\bPi_V$ is a projection.
\end{proof}

\section{Numerical Experiments}\label{sec-Numerics}
In this section we confirm
the theoretical results with some numerical experiments
on a variety of meshes and finite element spaces. All the numerical experiments were performed using {FEniCS} \cite{fenics}.
In the \revj{first four tests}, we take the domain to be the unit square $\Omega = (0,1)^2$.
The exact eigenvectors, corresponding to non-zero eigenvalues,
 are $\bu^{(n,m)}(x,y):= \bcurl p^{(n,m)}$ where $p^{(n,m)}:= \cos(\pi n x) \cos(\pi m y)$,
with eigenvalues 
$\lambda^{(n,m)}:= \pi^2 (n^2+m^2)$ for  $n,m\in \mathbb{N}\cup \{0\}$ and $nm\neq 0$.
In the following we relabel 
the non-zero eigenvalues $\lambda^{(i)}$ in non-decreasing order: $0 < \lambda^{(1)} \le  \lambda^{(2)} \le \lambda^{(3)}\le \cdots$

\subsection{Linear Lagrange elements on Powell--Sabin triangulations}
In these series of tests, we compute the finite element method
\eqref{fem} using piecewise linear Lagrange elements defined
on Powell--Sabin triangulations.
We create a sequence of generic Delaunay triangulations $\mct$ with mesh size $h_j = 2^{-j}$ for $j=3, 4, 5, 6$,
and perform the refinement algorithm described in Section \ref{sec-PS} to 
obtain a Powell--Sabin triangulation $\PST$ for each mesh parameter.

In Table \ref{table1}, we show the first ten non-zero  
approximate eigenvalues and errors using method \eqref{fem} defined on $\PST$
for fixed $h=1/32$. In Table \ref{table2}, we list the rate of convergence of the first eigenvalue
with respect to $h$.  The tables show an absence of spurious eigenvalues which agrees
with the theoretical results, Theorems \ref{mainthm} and \ref{thm:PSMainThm}.
In addition, we observe an asymptotic quadratic rate of convergence for the computed eigenvalue.

\begin{table}[h]
\begin{center}
\begin{tabular}{|c|c|c|}\hline
   $i$ & $\lambda^{(i)}_h$& $|\lambda^{(i)}-\lambda^{(i)}_h|$ \\
\hline
1& 9.872556542826 & 2.952141736802360E-3 \\
2& 9.872647617226 & 3.043216136799032E-3 \\
3& 19.75126057536 & 1.205177318315975E-2 \\
4& 39.52514303832 & 4.672543396706175E-2 \\
5& 39.52979992791 & 5.138232355238159E-2 \\
6& 49.42354393173 & 7.552192628650545E-2 \\
7& 49.43033089264 & 8.230888719544538E-2 \\
8& 79.15457141878 & 1.977362100693938E-1  \\
9& 89.06160447391 & 2.351648641029839E-1 \\
10&89.07453060702& 2.480909972125715E-1\\
%
\hline
\end{tabular}
\end{center}
\vspace{0mm}
\caption{Approximate eigenvalues of \eqref{fem}
using the piecewise linear Lagrange finite element space
on a  Powell--Sabin triangulation.  The mesh parameter is $h = 2^{-5}$.}
\label{table1}
\end{table}

\begin{table}[h]
\begin{center}
\begin{tabular}{|c|c|c|}\hline
    $h$ & $|\lambda^{(1)}-\lambda_h^{(1)}|$ & rate \\
\hline
   $2^{-3}$  &  1.084194558097806E-1   &   \\ 
    $2^{-4}$ &  3.835460507298371E-2 &   1.8228 \\  
    $2^{-5}$ &  2.952141736802360E-3  &  1.8768 \\
    $2^{-6}$ &  7.488421347368046E-4   & 1.9790 \\
\hline
\end{tabular}
\end{center}
\vspace{0mm}
\caption{The rate of convergence 
with respect to $h$ of first non-zero eigenvalue using for Powell--Sabin split and the linear Lagrange finite element space.} \label{table2}
\end{table}

\subsection{Quadratic Lagrange elements on Clough--Tocher triangulations}
In this section, we compute the finite element method \eqref{fem}
using quadratic Lagrange elements defined on Clough--Tocher triangulations (cf.~Section \ref{sec-CT}).
As before, we create a sequence of meshes $\mct$ with $h_j = 2^{-j}\ (j=3,4,5,6)$,
and construct the Clough--Tocher refinement $\CTT$ by connecting the vertices
of each triangle in $\mct$ with its barycenter; see Figure \ref{fig:TC8}.

\begin{figure}
\begin{center}
\includegraphics[scale=0.5]{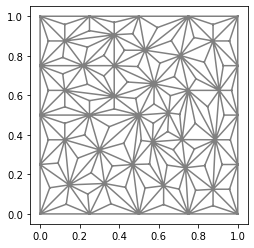}
\caption{A Clough--Tocher triangulation with $h =2^{-3}$.}
\end{center}
\label{fig:TC8}
\end{figure}

In Table \ref{table3} we report the first computed
 ten non-zero approximate eigenvalues using method \eqref{fem}.  
As predicted by Theorems \ref{mainthm} and  \ref{thm:CTmain}, 
the results show accurate approximations
with no spurious eigenvalues.   
 In Table \ref{table4} we list the rate of convergence to the first eigenvalue
 for different values of $h$.  The table clearly shows 
an asymptotic quartic rate of convergence: $|\lambda^{(1)}-\lambda_h^{(1)}| = \mathcal{O}(h^4)$.

\begin{table}[h]
\begin{center}
\begin{tabular}{|c|c|c|}\hline
    $i$ & $\lambda^{(i)}_h$& $|\lambda^{(i)}-\lambda_h^{(i)}|$ \\
\hline
1 & 9.869606458779 & 2.057689641788E-6 \\
2 & 9.869606625899 & 2.224809986018E-6 \\ 
3 & 19.73922733515 & 1.853298115861E-5 \\
4 & 39.47853970719 & 1.221028349079E-4 \\
5 & 39.47855143244 & 1.338280896661E-4 \\
6 & 49.34827341503 & 2.514095869017E-4 \\
7 & 49.34829772352 & 2.757180775106E-4 \\
8 & 78.95794423573 & 1.109027018615E-3 \\
9 & 88.82788915584 & 1.449546038714E-3 \\
10&88.82798471962 & 1.545109821734E-3 \\
\hline
\end{tabular}
\end{center}
\vspace{0mm}
\caption{Approximate eigenvalues using 
quadratic Lagrange elements on a Clough--Tocher triangulation with $h = 2^{-5}$.}\label{table3}
\end{table}

\begin{table}[h]
\begin{center}
\begin{tabular}{|c|c|c|}\hline
    $h$ & $|\lambda^{(1)}-\lambda_h^{(1)}|$ & rate \\
\hline 
$2^{-3}$    &   2.98012061403341E-4    &    \\ 
$2^{-4}$    &   2.96722579697928E-5    & 3.3282\\  
$2^{-5}$   &    2.05768964178787E-6  	& 3.8500   \\
$2^{-6}$   &    1.43249797801559E-7    & 3.8444  \\
\hline
\end{tabular}
\end{center}
\vspace{0mm}
\caption{The rate of convergence  of first non-zero eigenvalue using  the Clough--Tocher split and $k=2$} \label{table4}
\end{table}


\subsection{Quartic Lagrange elements on criss-cross triangulations}
In this section we compute the finite element method \eqref{fem}
using fourth degree Lagrange elements on several types of triangulations.
Theorems \ref{mainthm} and \ref{thm:SVmain} indicate
that this scheme leads to convergent eigenvalue approximations
as $h\to 0$ if the quantity $\Theta_{\min}$ is uniformly bounded from below.
We recall that the quantity $\Theta_{\min}$
gives a measurement of the closest to singular vertex in the mesh, i.e.,
$\Theta_{\min}$ is small if there exists a vertex in $\mct$
that falls on two ``almost'' straight lines; see \eqref{eqn:Thetamin} and
\eqref{eqn:ThetaDef} for the precise definition.

In the first series of tests, we numerical study the effect
of $\Theta_{\min}$ in the finite element method \eqref{fem}.
To this end, we first take $\mct$ to be 
the criss-cross mesh with $h=1/6$ (cf.~Figure \ref{fig:cc61}).
This triangulation has $36$ singular vertices, but
$\Theta_{\min}$ is well-behaved.   Theorems 
\ref{mainthm} and \ref{thm:SVmain} indicate
that the finite element scheme \eqref{fem} (with quartic Lagrange elements)
leads to accurate approximations.  Indeed, 
Table \ref{table5} lists the first ten computed non-zero eigenvalues,
and it clearly shows accurate results.

Next, we perform the same tests but randomly perturb 
each singular vertex of the criss-cross mesh by 
a factor $\alpha h$ for some $\alpha\in (0,1]$.  In particular, for each singular vertex $z\in \calS_h$
of the criss-cross triangulation $\mct$, we make the perturbation $z\to z+(\pm \alpha h,\pm \alpha h)$.
Figures \ref{fig:cc61}(right), \ref{fig:cc62}(left), and \ref{fig:cc62}(right)
show the resulting triangulations with $\alpha = 0.01$, $\alpha=0.05$, and $\alpha=0.1$, respectively.
We note that on the resulting perturbed mesh, $\Theta_{\min}\approx \alpha$,
and therefore Theorem \ref{thm:SVmain} suggests that the finite element approximation
\eqref{fem} may suffer for small $\alpha$-values.

The computed eigenvalues, with values $\alpha=0.01$, $\alpha=0.05$,
and $\alpha=0.1$, are reported in Tables \ref{table6}, \ref{table7}, and \ref{table8}, 
respectively.  Table \ref{table8} shows that, for relatively large perturbations ($\alpha=0.1$),
we compute relatively accurate eigenvalue approximations with similar convergence
properties found on the criss-cross mesh (cf.~Table \ref{table5}).
On the other hand, for smaller perturbations ($\alpha=0.05$ and $\alpha=0.01$),
the results drastically differ.  Table \ref{table6} clearly show
extremely poor approximations for all eigenvalues,
and Table \ref{table7} only computes the first few eigenvalues
with reasonable accuracy before the results deteriorate.  These numerical tests indicate the approximation
properties of the computed eigenvalues are highly sensitive to the quantity $\Theta_{\min}$.

\begin{figure}
\begin{center}
\includegraphics[scale=0.4]{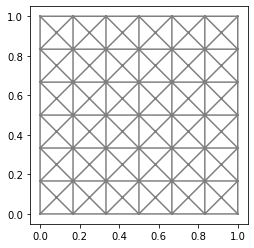} \quad 
\includegraphics[scale=0.4]{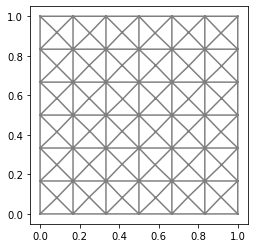} 
\caption{Left: Criss-cross mesh with $h = 1/6$.  Right:
The mesh obtained by randomly perturbing the singular vertices
of the criss-cross mesh by $0.01h$.}
\end{center}
\label{fig:cc61}
\end{figure}

\begin{figure}
\begin{center}
\includegraphics[scale=0.4]{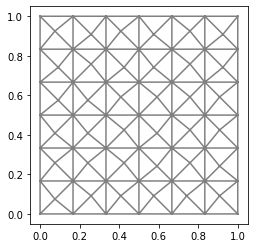} \quad 
\includegraphics[scale=0.4]{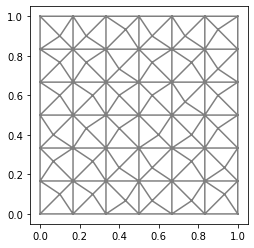} 
\caption{Criss-cross meshes with singular vertices randomly perturbed by $0.05 h$ (left)
and $0.1h$ (right).}
\end{center}
\label{fig:cc62}
\end{figure}

\begin{table}[h]
\begin{center}
\begin{tabular}{|c|c|c|}\hline
    $i$ & $\lambda^{(i)}_h$& $|\lambda^{(i)}-\lambda_h^{(i)}|$ \\
\hline
1 & 9.869604401309 & 2.199112003609E-10 \\
2 & 9.869604401309 & 2.200408744102E-10 \\ 
3 & 19.73920880459 & 2.414715538634E-09 \\
4 & 39.47841782951 & 2.251546860066E-07 \\ 
5 & 39.47841782951 & 2.251547499554E-07 \\ 
6 & 49.34802238840 & 3.829525141441E-07 \\ 
7 & 49.34802238840 & 3.829534165334E-07 \\ 
8 & 78.95683762620 & 2.417486058448E-06 \\ 
9 & 88.82645223886 & 1.262905662713E-05 \\
10&88.82645223886 & 1.262905958299E-05 \\ 
\hline
\end{tabular}
\end{center}
\vspace{0mm}
\caption{Approximate eigenvalues using quartic Lagrange elements on a criss-cross mesh with $h=1/6$.}\label{table5}
\end{table}

\begin{table}[h]
\begin{center}
\begin{tabular}{|c|c|c|}\hline
    $i$ & $\lambda^{(i)}_h$& $|\lambda^{(i)}-\lambda_h^{(i)}|$ \\
\hline
1 & 1.424154538647 & 8.445449862442 \\ 
2 & 1.471404605901 & 8.398199795188 \\
3 & 1.477776343297 & 18.26143245888 \\
4 & 1.502342236815 & 37.97607536754 \\ 
5 & 1.526468793982 & 37.95194881038 \\ 
6 & 1.540736126805 & 47.80728587864 \\
7 & 1.552154885100 & 47.79586712035 \\ 
8 & 1.556952619119 & 77.39988258960 \\ 
9 & 1.566640464185 & 87.25979914562 \\ 
10&1.580713040988 & 87.24572656882 \\
\hline
\end{tabular}
\end{center}
\vspace{0mm}
\caption{Approximate eigenvalues using 
quartic Lagrange elements on a $0.01h$-perturbed criss-cross mesh with $h = 1/6$.}\label{table6}
\end{table}

\begin{table}[h]
\begin{center}
\begin{tabular}{|c|c|c|}\hline
    $i$ & $\lambda^{(i)}_h$& $|\lambda^{(i)}-\lambda_h^{(i)}|$ \\
\hline
1 & 9.869604401311 & 2.212932059820E-10  \\
2 & 9.869604401311 & 2.215134742301E-10 \\ 
3 & 19.73920880479 & 2.614239491550E-09 \\ 
4 & 35.63498774612 & 3.843429858239 \\
5 & 36.48359498561 & 2.994822618752 \\ 
6 & 36.92351459416 & 12.42450741128 \\ 
7 & 37.63299206644 & 11.71502993900 \\ 
8 & 37.78514981304 & 41.17168539568 \\ 
9 & 38.10084364520 & 50.72559596460 \\
10&38.35191236801 & 50.47452724179 \\ 
\hline
\end{tabular}
\end{center}
\vspace{0mm}
\caption{Approximate eigenvalues using 
quartic Lagrange elements on a $0.05h$-perturbed criss-cross mesh with $h = 1/6$.}\label{table7}
\end{table}

\begin{table}[h]
\begin{center}
\begin{tabular}{|c|c|c|}\hline
    $i$ & $\lambda^{(i)}_h$& $|\lambda^{(i)}-\lambda_h^{(i)}|$ \\
\hline
1 & 9.869604401320 & 2.310134306071E-10 \\ 
2 & 9.869604401320 & 2.312834368468E-10 \\
3 & 19.73920880546 & 3.285371974471E-09 \\ 
4 & 39.47841784038 & 2.360199999885E-07 \\ 
5 & 39.47841784071 & 2.363495781310E-07 \\ 
6 & 49.34802242662 & 4.211773898533E-07 \\
7 & 49.34802246288 & 4.574410894520E-07 \\ 
8 & 78.95683842488 & 3.216167357323E-06 \\ 
9 & 88.82645270371 & 1.309390694360E-05 \\
10&88.82645276747 & 1.315766178323E-05 \\ 
\hline
\end{tabular}
\end{center}
\vspace{0mm}
\caption{Approximate eigenvalues using 
quartic Lagrange elements on a $0.1h$-perturbed criss-cross mesh with $h = 1/6$.}\label{table8}
\end{table}

\subsection{Quartic Lagrange elements on generic triangulations}
Our next series of tests
compute the finite element method \eqref{fem}
using quartic Lagrange elements on generic Delaunay triangulations.
Again, Theorem \ref{thm:SVmain} and the previous set of tests
indicate the approximation properties of the computed eigenvalues
are highly sensitive to the quantity $\Theta_{\min}$.
In light of this, for a given (generic) triangulation $\mct$
we randomly move each interior vertex
with four neighboring triangles by a $0.1h$-perturbation; see Figure \ref{fig:cc61B}.

Table \ref{tab:GSVP} shows the maximum errors of the first $20$
computed eigenvalues 
on these perturbed mesh for $h = 2^{-j}\ (j=2,3,4,5)$.  The table clearly 
shows convergence with rate $\mathcal{O}(h^8)$.
On the other hand, the errors of the computed eigenvalues on
`non-perturbed' meshes do not converge as shown in Table~\ref{tab:GSVNonP}. 

It is interesting to note that Costabel and Dauge \cite{CoDa02} showed that using quartics one has convergence on  any mesh if they use they add a $(\dive, \dive)$ stabilization term to the formulation (at least for convex polygons). However, here we see that the results are more sensitive with the formulation \eqref{weak}.

\begin{figure}
\begin{center}
\includegraphics[scale=0.4]{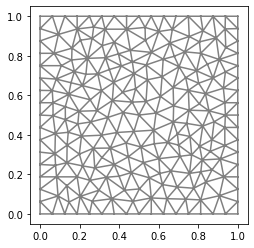} \quad 
\includegraphics[scale=0.4]{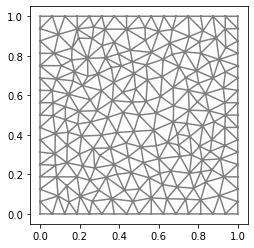} 
\caption{(left) Unstructured mesh with $h \approx 1/10$, (right) randomly perturbing interior vertices who have four triangles by at most $.1h$}
\end{center}
\label{fig:cc61B}
\end{figure}

%

\begin{table}[h]
\begin{center}
\begin{tabular}{|c|c|c|}\hline
    $h$ & $\max_{1\le i\le 20}|\lambda^{(i)}-\lambda_h^{(i)}|$ & rate \\
\hline 
$2^{-2}$ &   8.38611345105E-03 & \\
$2^{-3}$    &5.61831120933E-05 & 7.2217\\
$2^{-4}$    &59.2176263988 & -20.008\\
$2^{-5}$   & 59.2176264065 & 0.000\\
\hline
\end{tabular}
\end{center}
\vspace{0mm}
\caption{Maximum error of the first $20$ eigenvalues
on (non-perturbed) Delaunay triangulations using quartic Lagrange elements.
Note that for $h=2^{-2}$ and $h=2^{-3}$, the mesh $\mct$
does not have any vertices with four neighboring triangles.}\label{tab:GSVNonP}
\end{table}

\begin{table}[h]
\begin{center}
\begin{tabular}{|c|c|c|}\hline
    $h$ & $\max_{1\le i\le 20}|\lambda^{(i)}-\lambda_h^{(i)}|$ & rate \\
\hline 
$2^{-2}$ &    8.3861134511E-03  &\\
$2^{-3}$    & 5.6183112093E-05 & 7.2217\\
$2^{-4}$    & 2.2360291041E-07 & 7.9731\\
$2^{-5}$   &  8.9832496997E-10 & 7.9595\\
\hline
\end{tabular}
\end{center}
\vspace{0mm}
\caption{Maximum error of the first $20$ eigenvalues
on perturbed Delaunay triangulations using quartic Lagrange elements.}
\label{tab:GSVP}
\end{table}

\revj{
\subsection{L-Shaped Domains}
In this example, we consider an $L$-shaped domain: \\ $\Omega=[-\pi, \pi]^2 \backslash \Big([0, \pi] \times [-\pi, 0] \Big)$ . The first non-zero eigenvalue corresponds to an eigenvector that is not in $H^1$ and the approximate value of this eigenvalue is given by $\lambda^{(1)} \approx 0.149511749824251$ \cite{DaugeBenchMark}.  In Table \ref{tableL4} we give the error using Lagrange elements with $k=1$ on Powell-Sabin splits. In Table \ref{tableL5} we give the error using N\'ed\'elec elements of the second kinds with $k=1$ on the same meshes. As we can see, the rate of convergence seems to be tending to $4/3$ for both finite elements. The next eigenvalues correspond to eigenvectors that belong to $H^1$
and the convergence rates increase to $2$ for both elements, but we do not present the errors here.

We would like to stress that although the eigenvalues do converge as we proved, the convergence of the eigenvectors   will not converge in $H(\text{div}) \cap H(\text{curl})$ if the eigenfunctions are not in $H^1$. Instead convergence should be sought in the $H(\text{curl})$ norm. 
}

\begin{table}[h]
\begin{center}
\begin{tabular}{|c|c|c|}\hline
    $h$ & $|\lambda^{(1)}-\lambda_h^{(1)}|$ & rate \\
\hline 
$2^{-3}$    &       5.29957E-03  &    \\ 
$2^{-4}$    &     2.42718E-03   & 1.12659499\\  
$2^{-5}$   &        1.07087E-03	&  1.18049541   \\
$2^{-6}$   &        5.5788E-04  & 0.94076660  \\
$2^{-7}$   &       1.8099E-04    &   1.62402935    \\
$2^{-8}$   &       7.273E-05   & 1.31537097 \\
\hline
\end{tabular}
\end{center}
\vspace{0mm}
\caption{$L$-shaped domain: The rate of convergence  of first non-zero eigenvalue using  the Powell-Sabin split and $k=1$} \label{tableL4}
\end{table}

\begin{table}[h]
\begin{center}
\begin{tabular}{|c|c|c|}\hline
    $h$ & $|\lambda^{(1)}-\lambda_h^{(1)}|$ & rate \\
\hline 
$2^{-3}$    &       9.564E-05    &    \\ 
$2^{-4}$    &      6.285E-05  & 0.60567278 \\  
$2^{-5}$   &    3.039E-05  	&   1.04839118   \\
$2^{-6}$   &      1.763E-05 & 0.78534429 \\
$2^{-7}$   &        5.72E-06  &   1.62460510   \\
$2^{-8}$   &       2.41E-06    & 1.24527844 \\
\hline
\end{tabular}
\end{center}
\vspace{0mm}
\caption{$L$-shaped domain:  The rate of convergence  of first non-zero eigenvalue using  the N\'ed\'elec elements of the second kind and $k=1$} \label{tableL5}
\end{table}

\rev{
\section{Concluding Remarks}\label{sec-conclude}
In this paper, we studied
and numerically verified
the use of Lagrange finite element
spaces for the two-dimensional Maxwell eigenvalue
problem.  Using and extending
the analysis of divergence--free Stokes
pairs, we showed, on certain triangulations,
convergence of the discrete eigenvalues.

While the focus of this paper has been on the
two-dimensional setting, the tools developed here
may apply to three dimensions as well.
 In particular,
smooth, discrete de Rham complexes using Lagrange 
finite element spaces have been constructed in \cite{FGN20,GLN20},
and these results might be applicable to the 3D Maxwell 
eigenvalue problem.
}
\clearpage

\appendix
\section{{Proof of Lemma \ref{lem:SZ}}}
In order to describe the new interpolant, we first remind the reader of the Scott-Zhang interpolant \cite{ScottZhang90}. 
For every $z\in \calV_h$ we define $\phi_z \in \pol_1^c(\mct)$ to be the  hat function $\phi_z(y)=\delta_{yz}$ for all $y \in \calV_h$. Also  for every $z\in \calV_h$,  we identify an arbitrary edge $e_z$  of the mesh that contains $z$ with the only constraint that $e_z$ is a boundary edge if $z$ is a boundary vertex. Then there exists function $\psi_z \in L^\infty(e_z)$ such that
\begin{alignat}{1}
\int_{e_z} \psi_z \phi_y = \delta_{yz}  \qquad y \in \calV_h. \label{phiz}
\end{alignat}
Moreover, 
\begin{equation} \label{phizbound}
\|\psi_z\|_{L^\infty(e_z)} \le \frac{C}{\revj{|e_z|}}.
\end{equation}

The Scott-Zhang interpolant $\tilde{\bI}_h$ is given by: 
\begin{equation}\label{SZorg}
\tilde{\bI}_h \btau(x) = \sum_{z\in \calV_h}  \revj{(\int_{e_z} \psi_z \btau)} \phi_z(x).
\end{equation}

Although the Scott-Zhang interpolant has the approximation properties we need, it might not preserve the tangential trace to be zero. More precisely,  if $\btau \in \bH^{\frac{1}{2}+\delta}(\Omega)\cap \bH_0({\rm rot};\Omega)$
then  $\tilde{\bI}_h \btau \cdot \bt$ might not vanish on edges that touch a corner vertex.  Therefore, we must modify the Scott-Zhang interpolant on such vertices. 

For every corner boundary vertex $z \in  \calV_h^C$ we will consider the two boundary edges, $e_z^1,e_z^2$, that that have $z$ as an endpoint.
We let $\bn_z^i$ be the outward pointing normal to  $e_z^{i}$ and $\bt_z^i$ the tangent vector to $\bn_z^i$  that is rotated 90 degrees counterclockwise.   We then have the existence of  $\psi_z^i \in L^\infty(e_z^i)$ such that  
\begin{alignat}{1}
\int_{e_z^i} \psi_z^i \phi_y =& \delta_{yz}  \qquad y \in \calV_h,   \label{phizi}\\
\|\psi_z^i\|_{L^\infty(e^i_z)} \le & \frac{C}{\revj{|e_z|}}. \label{phizboundi}
\end{alignat}

We can then define the modified Scott-Zhang interpolant as follows: 
\begin{equation}\label{eqn:MSZ}
\bI_h \btau(x) := \sum_{z\in \calV_h\backslash \calV_h^C}\big(\int_{e_z} \psi_z \btau\big)  \phi_z(x)
+ \sum_{z\in \calV_h^C}\bbeta_z(\btau) \phi_z(x),
\end{equation}
where 
\begin{equation*}
\bbeta_z(\btau):= \frac{\bn_z^2}{\bn_z^{2} \cdot \bt_z^{1}} \int_{e_z^{1}} (\btau \cdot \bt_z^{1}) \psi_z^{1}
+  \frac{\bn_z^1}{\bn_z^{1} \cdot \bt_z^{2} } \int_{e_z^{2}} (\btau \cdot \bt_z^{2}) \psi_z^{2}.
\end{equation*}
We now proceed to prove Lemma \ref{lem:SZ} in four steps.\medskip


\noindent{\em (i) $\bI_h:\bH^{\frac{1}{2}+\delta}(\Omega)\cap \bH_0({\rm rot};\Omega)\to \bpol^c_1(\mct)\cap \bH_0({\rm rot},\Omega)$}:
If $ \btau \in \bH^{\frac{1}{2}+\delta}(\Omega) \cap \bH_0({\rm rot};\Omega)$, then clearly $\bI_h \btau(z)=0$ for every $z \in \calV_h^C$. Also we have   
$\bI_h \btau(z) \cdot \bt_z=0$ for all  $z \in \calV_h\backslash \calV_h^C$ where $\bt_z$ is tangent to $e_z$. Thus we have that $\bI_h \btau \cdot \bt=0$ on $\partial \Omega$.\medskip

\noindent{\em (ii) $\bI_h$ is a projection}:
In order to show it is a projection we  need to show that $\bI_h \btau(z)=\btau(z)$ for all $z \in  \calV_h$ and $\btau \in \bpol_1^c(\mct)$. To this end,  let $\btau \in \bpol_1^c(\mct)$.  If $z \in \calV_h\backslash \calV_h^C$ then $\bI_h \btau(z)=  \int_{e_z} \psi_z \btau$. However,  
$\int_{e_z} \psi_z \btau =\btau(z)$ by \eqref{phiz}, since $\btau|_{e_z}=\btau(z) \phi_z+ \btau(y) \phi_y$ where $y$ is the other end point of $e_z$. On the other hand, if $z \in \calV_h^C$ then  $\bI_h \btau(z)=  \bbeta_z(\btau)$.   Then, we have $\bbeta_z(\btau) \cdot \bt_z^i= \int_{e_z^{i}} (\btau \cdot \bt_z^{i}) \psi_z^{i}$. Using \eqref{phiz} we have that $\int_{e_z^{i}} (\btau \cdot \bt_z^{i}) \psi_z^{i}=\btau(z) \cdot \bt_z^{i}$.  Thus we have shown that $\bI_h \btau(z) \cdot \bt_z^i= \btau(z) \cdot \bt_z^{i}$ for $i=1,2$ and thus $\bI_h \btau(z)=\btau(z)$. \medskip

\noindent{\em (iii) Stability estimate}:
We derive a stability estimate following the arguments of \cite{ScottZhang90,PatCiarlet13}. First we note that by an inverse estimate we have
\begin{equation*}
|\bI_h \btau|_{H^{\frac{1}{2}+\delta}(T)} \le C h_T^{-\frac{1}{2}-\delta} \|\bI_h \btau\|_{L^2(T)}.
\end{equation*}
Thus, we only need to bound the $L^2$-norm.  To do this, we first note the trace inequality (cf.~\cite[Proposition 3.1]{PatCiarlet13})
\[
\|\btau\|_{L^1(e_z)}\le C\big(\|\btau\|_{L^2(T)} + h_{T}^{\frac{1}{2}+\delta}|\btau|_{H^{\frac{1}{2}+\delta}(T)}\big),
\]
for $T\in \mct$ with $e_z\subset \p T$. \revj{We remind the reader that the number of corner points $\calV_h^C$ is finite and independent of the mesh $\mct$ and hence  $M:=\max_{z\in \calV_h^C} \frac{1}{|\bn_z^{1} \cdot \bt_z^{2}|}$ is finite}. Thus, using \eqref{phizbound} and \eqref{phizboundi}, we have 
\begin{align*}
\|\bI_h \btau\|_{L^2(T)}
&\le  \mathop{\sum_{z\in \calV_h\backslash \calV_h^C}}_{z\in \bar T} \|\phi_z\|_{L^2(T)} \|\psi_z\|_{L^\infty(e_z)} \|\btau\|_{L^1(e_z)}\\
&\qquad+ M \mathop{\sum_{z\in  \calV_h^C}}_{z\in \bar T}  \|\phi_z\|_{L^2(T)} \big(\|\psi_z^{1}\|_{L^\infty(e_z^{1})} \|\btau\|_{L^1(e_z^{1})}
+\|\psi_z^{2}\|_{L^\infty(e_z^{2})} \|\btau\|_{L^1(e_z^{2})}\big)\\
&\le C(1+M)\big( \|\btau\|_{L^2(\omega(T))}+h_T^{\frac{1}{2}+\delta}|\btau|_{H^{\frac{1}{2}+\delta}(\omega(T))}\big),
\end{align*}
\revj{where we used that $\|\phi_z\|_{L^2(T)} \le C h_T$}. Hence,  combing the above results we obtain  
\begin{equation}\label{auxapp1}
  h_T^{\frac{1}{2}+\delta} |\bI_h \btau|_{H^{\frac{1}{2}+\delta}(T)}  + \|\bI_h \btau\|_{L^2(T)} \le  C(1+M)\big( \|\btau\|_{L^2(\omega(T))}+h_T^{\frac{1}{2}+\delta}|\btau|_{H^{\frac{1}{2}+\delta}(\omega(T))}\big).
\end{equation}

\noindent{\em (iv)  Estimate \eqref{eqn:SZmod}}:
Let $\bw=\frac{1}{|\omega(T)|} \int_{\omega(T)} \btau$, we have that
\begin{equation}\label{eqaux}
\|\btau-\bw\|_{L^2(\omega(T))} \le C h_T^{\frac{1}{2}+ \delta} |\btau|_{H^{\frac{1}{2}+\delta}(\omega(T))}.
\end{equation}
The estimate \eqref{eqaux} for $\delta = \frac{1}{2}$ is shown in \cite[Section 4]{ScottZhang90}.
\revj{The estimate \eqref{eqaux} for $\delta \in (0,\frac{1}{2})$ can be found for example in \cite[Proposition 2.1]{drelichman2017improved} and \cite[Lemma 3.1]{BellidoMora-Coral}. See also \cite[Lemma 5.6]{ErnGuermond}.}

Because $\bw$ is constant we have that $\bI_h \bw=\bw$ on $T$, and thus using  \eqref{auxapp1} and \eqref{eqaux} we obtain
\begin{alignat*}{1}
\|\bI_h \btau-\btau\|_{L^2(T)}=&\|\bI_h (\btau-\bw)+ (\bw-\btau)\|_{L^2(T)}  \\
\le  & C(1+M)\big( \|\btau-\bw\|_{L^2(\omega(T))}+h_T^{\frac{1}{2}+\delta}|\btau|_{H^{\frac{1}{2}+\delta}(\omega(T))}\big) \\
\le  & C(1+M)\big( h_T^{\frac{1}{2}+\delta}|\btau|_{H^{\frac{1}{2}+\delta}(\omega(T))}\big). 
\end{alignat*}
Similarly, $|\bI_h \btau|_{H^{\frac{1}{2}+\delta}(T)}=|\bI_h \btau-\bw|_{H^{\frac{1}{2}+\delta}(T)}=|\bI_h (\btau-\bw)|_{H^{\frac{1}{2}+\delta}(T)}$ and one can use  \eqref{auxapp1} and \eqref{eqaux} again to get 
$|\bI_h \btau|_{H^{\frac{1}{2}+\delta}(T)} \le  C(1+M) |\btau|_{H^{1/2+\delta}(\omega(T))}$.
This completes the proof of Lemma \ref{lem:SZ}.

\end{document}